\documentclass[11pt]{article}
\usepackage{amsmath,amsfonts,amssymb,amsthm}
\usepackage{color, graphicx}
\usepackage{enumitem}
\usepackage{rotating}
\usepackage{subfigure}
\usepackage{hyperref}

\definecolor{verde}{rgb}{0.0, 0.5,0.0}
 
\newtheorem{theorem}{Theorem}[section]
\theoremstyle{definition}
 
\newtheorem{remark}{Remark}[section]
\theoremstyle{definition}

\newtheorem{lemma}{Lemma}[section]

\usepackage{array,arydshln}

\newcommand\bb {\mathbf b}

\newcommand\bx {\mathbf x}

\newcommand\bA {\mathbf A}
\newcommand\bB {\mathbf B}

\newcommand\bH {\mathbf H}

\newcommand\indica {\mathbb{I}}

\newcommand\wc {\widehat{{c}}}

\newcommand\wbH {\widehat{\bH}}

\newcommand\itA {{\mathcal{A}}}

\newcommand\itF {{\mathcal{F}}}
\newcommand\itG {{\mathcal{G}}}
\newcommand\itH {{\mathcal{H}}}
\newcommand\itI {{\mathcal{I}}}

\newcommand\itK {{\mathcal{K}}}
\newcommand\itL {{\mathcal{L}}}
\newcommand\itM {{\mathcal{M}}}
\newcommand\itN {{\mathcal{N}}}

\newcommand\itS {{\mathcal{S}}}
\newcommand\itT {{\mathcal{T}}}
\newcommand\itU {{\mathcal{U}}}
\newcommand\itV {{\mathcal{V}}}

\newcommand\bbe {\mbox{\boldmath $\beta$}}

\newcommand\bbech {\mbox{\scriptsize${\bbe}$}}
\newcommand\betta {\mbox{\boldmath $\eta$}}

\newcommand\bgama {\mbox{\boldmath $\gamma$}}
\newcommand\bla {\mbox{\boldmath $\lambda$}}
\newcommand\blach {\mbox{\scriptsize\boldmath $\lambda$}}
\newcommand\blachch {\mbox{\tiny\boldmath $\lambda$}}
\newcommand\bmu {\mbox{\boldmath $\mu$}}

\newcommand\bnabla {\mbox{\boldmath $\nabla$}}

\newcommand\bnuch {\mbox{\scriptsize\boldmath $\nu$}}

\newcommand\bthe {\mbox{\boldmath $\theta$}}
\newcommand\bthech {\mbox{\footnotesize\boldmath $\theta$}}

\newcommand\bSi {\mbox{\boldmath $\Sigma$}}

\newcommand\walfa {\widehat{\alpha}}

\newcommand\wbeta {\widehat{\beta}}
\newcommand\wbbe {\widehat{\bbe}}

\newcommand\weta {\widehat{\eta}}

\newcommand\wkappa {\widehat{\kappa}}
\newcommand\wlam {\widehat{\lambda}}
\newcommand\wblam {\widehat{\bla}}
\newcommand\wmu {\widehat{\mu}}

\newcommand\wbnabla {\widehat{\bnabla}}

\newcommand\wsigma {\widehat{\sigma}}

\newcommand\wtheta {\widehat{\theta}}
\newcommand\wbthe {\widehat{\bthe}}

\newcommand\wtau {\widehat{\tau}}

\newcommand\wtbbe {\widetilde{\bbe}}

\newcommand\wtbthe {\widetilde{\bthe}}

\newcommand\wteta {\widetilde{\eta}}

\newcommand\wtlam {\widetilde{\lambda}}
\newcommand\wtblam {\widetilde{\bla}}

\newcommand\wtbthech {\widetilde{\bthech}}

\def\real{\mathbb{R}}
\def\natu{\mathbb{N}}
\def\qu{\mathbb{Q}}

\newcommand{\EME}{\mathbb{M}}
\newcommand{\esp}{\mathbb{E}}
\newcommand{\prob}{\mathbb{P}}

\newcommand{\var}{\mbox{\sc Var}}

\newcommand{\as}{\mbox{\footnotesize a.s.}}

\newcommand{\convpp}{ \buildrel{a.s.}\over\longrightarrow}

\newcommand{\convprob  }{ \buildrel{p}\over\longrightarrow}

\newcommand{\trasp}{^{\mbox{\footnotesize \sc t}}}

\newcommand\bcero {{\bf{0}}}

\def\dst{\displaystyle}

\def\median{\mathop{\mbox{median}}}

\def\mad{\mathop{\mbox{\sc mad}}}

\def\argmin{\mathop{\mbox{argmin}}}

\def\sd{\mathop{\rm SD}}
\def\MSE{\mathop{\rm MSE}}
\def\MISE{\mathop{\rm MISE}}

\newcommand{\identidad}{\mbox{\bf I}}

\newcommand\noi{\noindent}
\def\dst{\displaystyle}

\def\square{\ifmmode\sqr\else{$\sqr$}\fi}
\def\sqr{\vcenter{
         \hrule height.1mm
         \hbox{\vrule width.1mm height2.2mm\kern2.18mm
\vrule width.1mm}
         \hrule height.1mm}}

\newcommand{\eme}{\mbox{\scriptsize \sc m}}

\newcommand{\tuk}{\mbox{\scriptsize \sc t}}

\newcommand{\rob}{\mbox{\footnotesize \sc r}}

\newcommand\gm {\mbox{\footnotesize    \sc gm}}

\newcommand{\cl}{\mbox{\scriptsize \sc cl}}

\definecolor{verde}{rgb}{0.0, 0.5,0.0}

\begin{document}
 
\title{Robust estimators in a generalized partly linear regression model under monotony constraints}
\author{Graciela Boente$^1$, Daniela Rodriguez$^1$ and Pablo Vena$^1$ \\
 \footnotesize $^1$ Facultad de Ciencias Exactas y Naturales, Universidad de Buenos Aires and CONICET}
\date{}
\maketitle

%%%%%%%%%%%%%%%%%% ABSTRACT%%%%%%%%%%%%%%%%%%%%%%%%%%%%%%%%
\begin{abstract}
 In this paper, we consider the situation in which the observations follow an  isotonic generalized partly linear model. Under this model, the mean  of the responses is modelled, through a link function, linearly on   some covariates and nonparametrically on an univariate regressor in   such a way that the nonparametric component is assumed to be a   monotone function. A class of robust estimates for the monotone   nonparametric component and for the regression parameter, related   to the linear one, is defined. The robust estimators are based on a   spline approach combined with a score function which bounds large
  values of the deviance. As an application, we consider   the isotonic partly linear log--Gamma regression  model.   Through a Monte Carlo study, we investigate the performance of the proposed  estimators under a partly linear log--Gamma regression  model with increasing nonparametric component.

\end{abstract}

%\normalsize
\newpage

\normalsize
\section{Introduction}\label{intro}

As is well known, semiparametric models may be introduced when the linear model is insufficient to explain the relationship between the response variable and its associated covariates.  This approach has been used to extend generalized linear models to allow most predictors to be modelled linearly while one or a small number of them enter the model nonparametrically. In this paper, we deal with observations $(y_i,\bx_i\trasp, t_i)\trasp$ satisfying  a semiparametric generalized partially linear model, denoted \textsc{gplm}. To be more precise, we assume that  $y_i|(\bx_i,t_i)\sim F(.,\mu_i,\kappa_0)$ where
$\var(y_i|(\bx_i,t_i))=A^2(\kappa_0)\, V^2(\mu_i)$,  with $A$ and  $V$  known functions and $\mu_i=\esp(y_i|(\bx_i,t_i))=\mu\left(\bx_i,t_i\right)$ is such that
\begin{equation}
  \mu\left(\bx,t\right)=H\left( \bx\trasp \bbe_0+\eta_0(t)\right)\,,
  \label{eq:GPLM}
\end{equation}
where $H^{-1}$ is a known link function, $\bbe_0\in \real^p$ is an unknown parameter and $\eta_0$ is an unknown continuous function with support on a compact interval $\itI$, which  we will assume equal to $[0,1]$, without loss of generality. The parameter $\kappa_0$ which is usually a nuisance parameter, generally lies on a subset of $\real$, for that reason we will assume that $\kappa_0\in\itK$, where $\itK \subset \real$ stands for an open set.

When $H(t)=t$, the generalized partially linear model is simply the well known partly linear regression model, that has been considerably studied, and, in this case, $\kappa_0$ is the  scale parameter. We refer for instance to H\"ardle \textsl{et al.} (2000). Robust estimators for \textsc{gplm} have been considered for instance by Boente \textsl{et al.} (2006) and by Boente and Rodr\'iguez (2010). However, in this paper, we deal with the situation in which
there are constraints on the nonparametric component $\eta_0$. More precisely, we will assume that   $\eta_0$, in model \eqref{eq:GPLM}, is monotone and for simplicity and without loss of generality non--decreasing. Most studies on generalized partly linear models assume that $\eta_0$ is an unspecified smooth function. However, in many applications, monotonicity is a property of the function to be fitted. Some examples when $\bbe_0=\bcero$ can be found for instance in Ramsay (1988) who studied the relation between the incidence of Down's syndrome and the mother's age; see also He and Shi (1998). In   Section  \ref{cost}, we analyse a data set considered in Marazzi and Yohai (2004) which aims to study the relationship between the hospital cost of stay and several explanatory variables, including the length of stay in days which we model  non--parametrically. The monotone assumption on $\eta_0$ is natural in this data set,  since the hospital cost increases  the longer the stay.

 Most estimation developments under monotone constraints were given under a partly linear regression model and we can mention among others, Huang (2002), Sun \textsl{et al.}
(2012) who considered estimation under constraints and also Lu (2010) who proposed a sieve maximum likelihood estimator based on $B-$splines. Recently, Lu (2015) considered a spline approach to generalized monotone partial linear models. All these methods are sensitive to outliers and some developments were given under a regression model, that is, when $H(t)=t$ to provide robust estimators. For nonparametric isotonic regression models, He and Shi (1998) and Wang and Huang (2002) proposed a robust isotonic estimate procedure based on the median regression, while, to improve the efficiency,  \'Alvarez and Yohai (2012) considered $M-$estimators for isotonic regression. On the other hand, under a partly linear regression model and following the approach given by  Lu (2010),  Du \textsl{et al.} (2013) consider $M-$estimators based on monotone $B-$splines  when $\eta_0$ is assumed to be a monotone function, the scale parameter is known and the errors have a symmetric distribution.  However, in the hospital data set to be considered in Section  \ref{cost}, the errors follow an asymmetric log--Gamma distribution and the proposal considered in Du \textsl{et al.} (2013) is not appropriate. Furthermore, the shape parameter is unknown and needs to be estimated in order to calibrate the robust estimators and to downweight large residuals. 

In this paper, we provide a general setting   to provide a family of estimators for the regression parameter $\bbe_0$ and the monotone regression function $\eta_0$ under the \textsc{gplm} model \eqref{eq:GPLM} when the nuisance parameter is unknown. This model includes a partly linear isotonic regression model with unknown scale and a  partly linear isotonic log--Gamma regression model with unknown shape parameter, as particular cases. In this sense, we generalize the proposal given in Du \textsl{et al.} (2013) by considering a preliminary scale estimator. The paper is organized as follows. Section \ref{proposal} described the proposed robust estimators. In particular, since our approach is based on $B-$splines,  a data--driven robust selection method for the knots is described. Consistency and rates of convergence for the proposed estimators are given in Section \ref{consistency}. The particular case of the log--Gamma model is considered in Section \ref{gamacaso}, while in Section \ref{montecarlo}, a  numerical study is carried out  to examine the small sample properties of the proposed procedures. An application to a real data set is provided in Section \ref{cost},  while concluding remarks are given in Section \ref{sec:comment}.     Some comments regarding the Fisher--consistency of the proposed estimators are given in Appendix A, while the proofs of the main results are relegated to Appendix B.

\section{The robust estimators}{\label{proposal}}

Let $w:\real^p \to \real$ be a weight function to control leverage points on the carriers $\bx$ and $\rho:\real^2\to\real$ a loss function. Define the functions
\begin{eqnarray}
  L_n(\bbe,g,a) &= & \frac 1n \sum_{i=1}^n   \rho\left(y_i,\bx_i\trasp\bbe+g(t_i),a\right)  w(\bx_i) 
  \label{eq:objsample}\\
  L(\bbe,g,a) &= & \esp   \rho\left(y_1,\bx_1\trasp\bbe+g(t_1),a\right)   w(\bx_1) \;.
  \label{eq:objfunction}
\end{eqnarray}
As in Lu (2010, 2015) and  Du \textsl{et al.} (2013),  consider $\itT_n=\{t_i\}_{i=1}^{m_n+2\ell}$ where $0=t_1=\dots=t_\ell< t_{\ell+1}< \dots< t_{m_n+\ell+1}= \dots= t_{m_n+2\ell}=1$ is a sequence of knots that partition the closed interval $[0, 1]$ into $m_n+1$ subintervals $\itI_i = [t_{l+i}, t_{l+i+1})$, for $i = 0, \dots ,m_n-1$ and $\itI_{m_n} = [t_{m_n+\ell}, t_{m_n+\ell+1}]$. 

Denote as  $\itS_n(\itT_n,\ell)$  the class of splines of order $\ell>1$ with knots  $\itT_n$. According to Corollary 4.10 of Schumaker (1981), for any $g\in\itS_n(\itT_n,\ell)$, there exist a class of $B-$spline basis functions  $\{B_j: 1\leq j\leq k_n\}$, with $k_n=m_n+\ell$, such that $g=\sum_{j=1}^{k_n}\lambda_jB_j$. Furthermore, according to Theorem 5.9 of Schumaker (1981), the spline $g$ is monotonically nondecreasing on $[0, 1]$ if nondecreasing constraints are imposed on the coefficients $\bla=(\lambda_1,\dots,\lambda_{k_n})\trasp$, i.e., when $\lambda_1\le \dots\le \lambda_{k_n}$.

Therefore, we can define a   collection of monotone non-decreasing splines on $[0,1]$, $\itM_n({\cal T}_n,\ell)$, which is a subclass ${\cal S}_n({\cal T}_n,\ell)$, through 
$$ \itM_n(\itT_n,\ell)=\left\{\sum_{i=j}^{k_n}\lambda_j B_j:   \quad\lambda_1\leq \dots \leq \lambda_{k_n}\right\} \,, $$
where the non-decreasing constraints are imposed on the coefficients to guarantee monotonicity. Hence, the function $\eta_0$ can be approximated as
$\eta(t)\approx \bla\trasp\bB(t)$ with  $\bB(t)=(B_1(t), \dots, B_{k_n}(t))\trasp$ the vector of $B-$spline basis functions, $\bla=(\lambda_1,\dots,\lambda_{k_n})\trasp$ the spline coefficient vector such that $\bla\trasp\bB\in \itM_n(\itT_n,\ell)$.

This suggests that estimators of $(\bbe_0,\eta_0)$ may be obtained minimizing $L_n(\bbe,g, \wkappa)$ over $\bbe\in \real^p$ and  $g\in \itM_n(\itT_n,\ell)$, where $\wkappa$ is a robust consistent estimator of $\kappa_0$, for instance, previously computed without the monotonicity constraint.  More precisely, the estimators $(\wbbe,\weta)=(\wbbe, \sum_{j=1}^{k_n}\wlam_jB_j)=(\wbbe,\wblam\trasp \bB)$ are defined through the values  $(\wbbe,\wblam)$ such that
\begin{equation}
  (\wbbe,\wblam) = \argmin_{\bbech \in \real^p,\blach\in \itL_{k_n}} L_n\left(\bbe,\sum_{j=1}^{k_n}\lambda_j\,B_j, \wkappa\right) \,,
  \label{eq:estimadores}
\end{equation}
where $\itL_{k_n}=\{\bla\in \real^{k_n}: \lambda_1\leq \dots\leq \lambda_{k_n}\}$. If we denote $\bB_i = \left(B_1(t_i),\dots,B_{k_n}(t_i) \right)$, we have that
\begin{equation}
  (\wbbe,\wblam)  =  \argmin_{\bbech\in \real^p,\blach\in \itL_{k_n}}  \frac 1n \sum_{i=1}^n  \rho\left(y_i,\bx_i\trasp \bbe + \bB_i\trasp
    \bla, \wkappa \right) w(\bx_i) \;.
  \label{eq:minim}
\end{equation}
Let $\itG= \{g: g \mbox{ is  a monotonically nondecreasing function on } [0, 1]\}$.   Throughout the paper, we will assume Fisher--consistency, i.e., 
\begin{equation}
L(\bbe_0,\eta_0,\kappa_0)=\dst\min_{\bbech\in \real^p, g\in \itG}L(\bbe,g,\kappa_0)\,,
\label{fisher}
\end{equation}  
with $(\bbe_0, \eta_0)$ being the unique minimum, that is, $L(\bbe_0,\eta_0,\kappa_0)<L(\bbe,g,\kappa_0)$ for any $(\bbe,g)\in\real^p\times \itG$, $(\bbe,g)\ne  (\bbe_0,\eta_0)$. This is a usual condition in robustness and it states that our target are indeed the true parameters of the model. A similar condition  for generalized linear models was required in Bianco \textsl{et al.} (2013a) and for generalized partial linear models in Boente \textsl{et al.} (2006) and Boente and Rodr\'{\i}guez (2010) who provide conditions ensuring that  $L(\bbe_0,\eta_0,\kappa_0)=\dst\min_{\bbech\in \real^p, g\in \itG}L(\bbe,\eta,\kappa_0)$.

\begin{remark}{\label{rem:rem1}}
As mentioned in Lu (2015), if $\lambda_1\le \dots\le \lambda_{k_n}$, the  function $g=\sum_{j=1}^{k_n}\lambda_jB_j$ is non--decreasing, but the linear inequality constraint on   the coefficients   is not a necessary condition. However, for quadratic $B-$splines, the coefficients condition is sufficient and     necessary for monotonicity.
\end{remark}

\subsection{The loss function}  \label{sec:lossfunctions}
Under a fully parametric generalized linear model, the selected loss function $\rho$ aims to bound either large values of   the deviance or of the Pearson   residuals. We refer to   Bianco and Yohai (1996), Croux and Haesbroeck (2003), Bianco \textsl{et al.} (2005) and  Cantoni and Ronchetti (2001), where different choices for the loss function are given. On the other hand,  optimally bounded score functions have been studied in Stefanski \textsl{et al.} (1986).   We briefly remind the definition of the family which bounds the deviance  which is the function used in our simulation study, for more details see, for instance, Boente \textsl{et al.} (2006) who   considered this family of loss functions to estimate the parameters of a generalized partial linear model using a profile--kernel approach. 

%For families of distributions that can be transformed to avoid an extra parameter in the model, a class of loss functions can be defined as follows. 
Let $\varphi_a$ be a bounded non--decreasing function with continuous derivative $\varphi_a^{\prime}$, $a$ being the tuning constant. Typically, $\varphi_a$ is a function performing like the identity function in a neighbourhood of 0 but bounding large values of the deviance.    Denote as  $f(\cdot,s)$   the density of the   distribution function $F(\cdot,s)$ with $y|(\bx,t)\sim F\left(\cdot, H\left (\eta (t) +\bx\trasp\bbe\right)\right)$. In this setting, the  \textsl{robust deviance--based  estimator} are related to the following choice for the function $ \rho(y,u,a)$
\begin{equation}
 \rho(y,u,a) = \varphi_a [-\log \,f(y,H(u)) + \log \,f(y,y) ] + G_a(H(u))\;.
\label{eq:rhobia}
\end{equation}
The correction term $G_a$     is given by
$$
G_a^{\prime}(s) =%\int \varphi_a^{\prime} [-\log \,f(y,s) +  \log  f(y,y) ] \,f^{\prime}(y,s) d\mu(y)=
\esp_s\left (\varphi_a^{\prime} [-\log \,f(y,s) +  \log  f(y,y) ] \,\frac{f^{\prime}(y,s)}{f(y,s)}\right)\;,
 $$
where $\esp_s$ indicates expectation taken under $y\sim F(\cdot,s)$ and $f^{\prime}(y,s)$ is a shorthand for $\partial\,f(y,s)/{\partial s}$.    It is worth noticing that  $\varphi_a(s)=s$,  $G_a(u)=0$ and $w\equiv 1$ when considering  the maximum likelihood estimator, under a generalized linear model.  
For a general function $\varphi_a$,  the correction factor  is included to guarantee Fisher--consistency under the true model,  as for generalized linear models. If the correction factor is taken equal to $0$, the results stated in Section   \ref{consistency} only ensure that the estimators will be consistent to the minimizer $(\bbe_F,\eta_F)$ of $L(\bbe,g,\kappa_0)$,  where $L(\bbe,g,a)$ is defined in \eqref{eq:objfunction}.  However, as discussed in Bianco \textsl{et al.} (2005), when considering a continuous family of distributions with strongly unimodal density function, the
correction term $G_a$ can be avoided. In this case, $\kappa_0$ may play  the role of the tuning constant. For instance, for the Gamma distribution, the tuning constant depends on the shape parameter so, if the shape is unknown, initial estimators need to be considered. Further details are given in Section \ref{gamacaso}. 

Note that for the Poisson and logistic regression models, we have   $\kappa_0= 1$, so $\kappa_0$ does not need to be estimated, hence $\varphi_a(s)=\varphi(s)$. Furthermore, as noted by  Croux and Haesbroeck (2003) for the logistic model,  in order to guarantee existence of solution, beyond the overlapping condition required for the maximum likelihood estimator, the derivative $\varphi^{\prime}$ of the function   $ \varphi(s)$ must satisfy additional constraints. More precisely,  $\varphi^{\prime}$ needs to be   increasing on $(-\infty, A_0]$ and decreasing on $[A_0, +\infty)$ for some $A_0>0$ or   increasing on $\real$ and also to fulfil that $\lim_{s\to +\infty} \varphi^{\prime}(st)/\varphi^{\prime}(-s)=\infty$ for any $t>0$. An example of function $\varphi$ satisfying these conditions is also given therein.
 
On the other hand, when $H(u)=u$, the usual square loss function  is replaced by a $\rho-$function after scaling the residuals to control the effect of large responses.  More precisely, let   $\phi : \real \to [0, \infty)$ stands for $\rho-$function as defined in Maronna \textsl{et al.}  (2006), i.e., an even continuous, non-decreasing function with $\phi(0)=0$ and such that $\phi(u)<\phi(v)$ when $0 \leq u< v$ with $\phi(v) < \sup_s \phi(s)$. Then, when the link function equals to identity function  $\rho(y, u,a)=\phi((y-u)/a)$ and, as mentioned in the Introduction, $\kappa_0$ plays the role of the scale parameter.

\begin{remark}
\begin{enumerate} 
\item[a)] As noted in Boente \textsl{et al.} (2006), under a logistic partially linear regression model, Fisher--consistency can easily be derived for the loss function given by \eqref{eq:rhobia}, when $\varphi$  satisfies the regularity conditions stated in Bianco and Yohai (1996), $w(\bx)>0$, for all $\bx$, and 
\begin{equation}
\prob\left(\bx\trasp\bbe=a_0|t=t_0\right)<1, \qquad \forall  (\bbe,a_0)\ne 0 \mbox{ and for almost all $t_0$}.
\label{eq:identifia}
\end{equation}
  Moreover, taking conditional expectations with respect to $(\bx,t)$, it is easy to verify that $(\bbe_0, \eta_0)$ is the unique minimizer of
  $L(\bbe, g,\kappa_0)$ in this case.  Condition \eqref{eq:identifia} does not allow $\bbe_0$ to include an intercept, so that the model will be   identifiable. %More details are given in the Appendix.

\item[b)] Under a generalized partially linear model with responses having a gamma distribution, Theorem 1 of Bianco \textsl{et al.} (2005) allows us to derive Fisher--consistency for the nonparametric and parametric components, if the score function  is bounded and strictly increasing on the set where it is not constant and if \eqref{eq:identifia} holds (see  Section \ref{gamacaso}).

\item[c)] Finally, consider  the  partially  linear model   $y_i=\bx_i\trasp\bbe_0+\eta_0(t_i)+\epsilon_i$ where $\epsilon_i$ are independent of $(\bx_i,t_i)$, that is,  the link function equals $H(u)=u$. In this case, Fisher--consistency   holds if, for instance, the errors $\epsilon_i$ have a symmetric distribution with density strictly unimodal, the loss function  equals  $\rho(y, u,a)=\phi((y-u)/a)$  with $\phi$ a $\rho-$function as defined in
 Maronna \textsl{et al.}  (2006), i.e., an even continuous,
 non-decreasing function with $\phi(0)=0$ and such that
 $\phi(u)<\phi(v)$ when $0 \leq u< v$ with $\phi(v) < \sup_s
 \phi(s)$.  Furthermore, we also have that $L(\bbe_0,\eta_0,a)= \min_{\bbech\in \real^p, g\in \itG}L(\bbe,g,a)$, for any $a>0$,  see   Appendix A for a proof. 
\end{enumerate}
\end{remark}

\subsection{Selection of $k_n$}{\label{sec:BIC}}
%In this paper, we use cubic splines with $\ell=3$, but linear and quadratic splines can be used if we think that $\eta$ is less smooth. 
A remaining question is the choice of the number knots and their location for the space of $B-$splines. Knot selection is more important for the estimate of
$\eta_0$ than for the estimate of $\bbe_0$.
One approach is to use uniform knots which is the approach followed in our simulation study. Uniform knots are usually sufficient when the function $\eta_0$ does not exhibit dramatic changes in its derivatives. On the other hand, non--uniform knots are desirable when the function has very different local behaviours in different regions. Another commonly used approach is to consider as knots quantiles of the observed $t_i$ with uniform percentile ranks. 

The number of knots $m_n$ or equivalently the number of elements of the basis (recall that $k_n = m_n + \ell$) may be determined by a
model selection criterion. Suppose that $ (\wbbe^{(k)},\wblam^{(k)})$ is the estimator solution of \eqref{eq:estimadores} with a $k-$dimensional
spline space.  As in He and Shi (1996) and He \textsl{et al.} (2002), for each $k$ define a criterion analogous to Schwartz (1978) information criterion
$$BIC(k) =  \frac 1n \sum_{i=1}^n \rho\left(y_i,\bx_i\trasp\wbbe^{(k)}+\sum_{j=1}^{k}\wlam_j^{(k)}\,B_j(t_i), \wkappa\right)
  w(\bx_i)   +\frac{\log n}{2\,n}(k+p)\,.$$  
Large values of $BIC$ indicate poor fits. A robust version of the Akaike criterion considered in Lu (2015) can also be considered. As is usual in spline--based procedures the number of knots should increase slowly with the sample
size $n$ to attain an  optimal rate of convergence. When it is assumed that $\eta$ is twice continuously
differentiable and  cubic splines ($\ell=3$) are considered, as in our simulation study, according to the convergence rate derived in Theorem \ref{thm:rate}, a possible criterion is to search for the first  (i.e. smallest $k$) local
minimum of $BIC(k)$ in the range of $\max(n^{1/5}/2, 4) \le k\le 8 + 2\, n^{1/5}$. Within this range, there is usually only one local minimum. The reason for $k$  being larger than $4$ is that for cubic splines the smallest possible choice is $4$. Also note that the global minimum of $BIC(k)$ actually occurs at a saturated model in which $k = n-p$, so $BIC(k)$ is a valid criterion only for a limited range of $k$.

\section{Consistency}\label{consistency}
In this section, we will derive, under some regularity conditions,  consistency and rates of convergence for the estimators defined in the previous Section. We will
begin by fixing some notation.  Let $\|\cdot\|$ the Euclidean norm of $\real^p$ and $\| f\|^2_2=\left(\esp f^2(t_1)\right)^{1/2}$.   For any continuous function
$v:\real\to \real$ denote $\|v\|_{\infty}= \sup_{t}|v(t)|$ and $\itG=\{g: g \text{ is  a monotonically nondecreasing function on } [0,
1]\}$.  From now on, $\itV$ stands for a neighbourhood of  $\kappa_0$ with closure  $\overline{\itV}$ strictly included in $ \itK$ and $\itF_n$ will denote the family of functions
$$\itF_n=\{f(y,\bx,t)=\rho\left(y,\bx\trasp\bbe+\bla\trasp\bB(t), a\right)w(\bx), \bbe \in \real^p, \bla\in \itL_{k_n}, a\in  \itV  \}\,.$$
Furthermore, for any measure $Q$, $N(\epsilon, \itF_n, L_s(Q))$ and $N_{[\;]}(\epsilon, \itF_n, L_s(Q))$ stand  for the covering  and bracketing  numbers  of the class $\itF_n$ with respect to the distance in $ L_s(Q)$, as defined, for instance, in van der Vaart and Wellner (1996).

\subsection{Consistency results}
 To derive the consistency of our proposal in the general framework we are considering, we will need the following set of assumptions whose validity is discussed in Remark \ref{rem:supC0}.

\begin{itemize}
\item[\textbf{C0.}] The estimators $\wkappa$ of $\kappa_0$ are strongly consistent.
\item[\textbf{C1.}]    $\rho(y,u,a)$ and $w( \cdot )$ are non--negative and  bounded functions and  $\rho(y,u,a)$ is a continuous function. Moreover,   $L^{\star}(\bbe, \bla, a)=L(\bbe, \sum_{j=1}^{k_n}\lambda_jB_j, a)$ satisfies the following equicontinuity condition:  for any $\epsilon>0$ there exists $\delta>0$ such that for any $a_1, a_2 \in \overline{\itV}$,
$$|a_1-a_2|<\delta \Rightarrow \sup_{{\bbech\in \real^k, \blach \in \itL_{k_n}}}|L^{\star}(\bbe, \bla, a_1)-L^{\star}(\bbe, \bla, a_2)|<\epsilon\, .$$

\item[\textbf{C2.}] The  true function $\eta_0$ is nondecreasing and its $r-$th derivative satisfies a Lipschitz condition on $[0, 1]$, with
  $r\geq 1$, that is,
$$\eta_0 \in \itH_r = \{ g\in C^r[0, 1]: \|g^{(j)}\|_\infty\leq C_1,\;0\leq j\leq r\; \mbox{ and } \; |g^{(r)}(z_1) - g^{(r)}(z_2)| \leq C_2|z_1-  z_2|\}\,.$$

\item[\textbf{C3.}] The maximum spacing of the knots is assumed to be of order $O(n^{-\nu})$, $0< \nu < 1/2$. Moreover, the
  ratio of maximum and minimum spacings of knots is uniformly bounded.
  
\item[\textbf{C4.}]  The class of functions   $\itF_n$ is such that, for any   $0<\epsilon<1$, $\log\left(N(\epsilon, \itF_n, L_1(P_n))\right)= O_\prob(1)({k_n})\log(1/\epsilon)$, for some constant $C_1>0$ independent of $n$  and $\epsilon$.

\end{itemize}

For simplicity, denote as $L(\bthe_0, \kappa_0)=L(\bbe_0, \eta_0, \kappa_0)$, where $\bthe_0=(\bbe_0, \eta_0)$ and $ \wbthe=(\wbbe, \weta)$  the estimators defined through \eqref{eq:estimadores} with $\weta(t)=\sum_{j=1}^{k_n}\wlam_j\,B_j(t )$. To measure the closeness between the estimators and the parameters, consider the metric $\pi^2(\bthe_0 ,\wbthe ) =\|\bbe_0-\wbbe \|^2+ \|\eta_0-\weta\|^2_{\itF}$ where $\|\cdot\|_{\itF}$ stands for a norm in the space of functions $\itF=\{g\,:\,[0,1]\to \real, \mbox{ such that } g \mbox{ is a continuous function} \}$, such as $\| f\|_2=\left(\esp f^2(t_1)\right)^{1/2}$ or    $\|f\|_{\infty}=\sup_{t\in [0,1]} |f(t)|$. Let $\itA_\epsilon=\{\bthe=(\bbe,g): \bbe\in \real^p, g \in \itG\cap \itF, \pi(\bthe,\bthe_0)>\epsilon\}$.

% Teorema de consistencia
\begin{theorem} \label{thm:consistency}
  Let $(y_i,\bx_i,t_i)\trasp$ be i.i.d. observations satisfying \eqref{eq:GPLM}.  Assume that \textbf{C0}  to \textbf{C4} hold and that for any $\epsilon>0$, $\inf_{\bthech\in \itA_\epsilon}L(\bthe,\kappa_0)>L(\bthe_0,\kappa_0)$ and that $k_n = O(n^\nu)$ for $1/(2r + 2) < \nu <   1/(2r)$. Then, we have that $\pi(\bthe_0 ,\wbthe )\convpp 0$.
\end{theorem}

% Remark sobre los supuestos
\begin{remark}{\label{rem:supC0}}
 As mentioned above, for the logistic and Poisson model, $\kappa_0$ is known and does not need to be estimated, hence  \textbf{C0} may be omitted. On the other hand, when $H(t)=t$ the scale parameter $\kappa_0$ may be estimated using any robust scale estimator computed without using the monotone constraint. To be more precise, let  $(\wbbe,\weta )$ be the robust estimators of $(\bbe_0,\eta_0)$ defined in Bianco and Boente (2004) and define the residuals as $r_i=y_i- \bx_i\trasp \wbbe-\weta(t_i)$. The scale estimator $\wkappa$ can be taken as  $\median_{1\le i\le n} |r_i|$.  Another possibility is to consider a   scale estimator based on a $\rho-$function as follows.  As in Maronna \textsl{et al.} (2006), let $\chi : \real \to \real_+$ be a $\rho-$function, that is, an even function, non--decreasing on $|t|$, increasing for $t>0$ when $\chi(t)<\|\chi\|_\infty$ and such that $\chi(0) = 0$. The   estimator $\wkappa$  of the scale $\kappa_0$  is the solution    
\begin{equation}\label{escala}
\frac 1n\sum_{i=1}^n  \chi_c \left(  \frac{ r_i}{s} \right) \ = b \,,
\end{equation}
 where $\chi_c(u) = \chi(u/c)$,   $c > 0$ is a user--chosen tuning constant and $b$ is related to the breakdown point of the scale estimator. If  $\chi$ is bounded, it is usually assumed that $\|\chi\|_{\infty}=1$ in which case  $0<b<1$. For instance, when $\chi$ is the Tukey's biweight function, the choice
$c = 1.54764$ and $b=1/2$ leads to an scale estimator  Fisher--consistent at the normal distribution with breakdown point $0.5$. On the other hand,   the choice $\chi(t)=\indica_{(1,\infty)}(|t|)$, $c=1$ and $b=0.5$ leads to   $\median_{1\le i\le n} |r_i|$. Similarly,  when the responses have a Gamma distribution the parameter $\kappa_0$ corresponds to the tuning constant and is related to the shape parameter. It can be estimated using a preliminary $S-$estimator  computed without making use of the monotone restriction, as   described in  Section \ref{gamacaso}. Straightforward calculations allow to show that in both situations \textbf{C0} holds. 

Assumption   \textbf{C1} is a standard requirement since it states that the weight function controls large values of the covariates and that the score function bounds large residuals, respectively. Moreover, the equicontinuity requirement allows to deal with the nuisance parameter in a general setting and a similar condition appears in   Bianco \textsl{et al.} (2013a). For the particular case of a partly linear regression model, i.e., when $H(t)=t$, $\kappa_0$ is the scale parameter and the function $\rho(y,u,a)$ is usually chosen as $\rho(y,u,a)=\phi((y-u)/a)$ where the function $\phi$ is an even, bounded  function, non--decreasing on $(0,\infty)$. In this case, the equicontinuity condition is satisfied, for instance, if $\phi$ is continuously differentiable with first derivative $\phi^{\prime}$ such that $s \, \phi^{\prime}(s)$ is bounded.

\textbf{C2} and \textbf{C3} are conditions regarding the smoothness of the nonparametric component and the knots spacing. They are analogous to those considered, for instance, in Lu (2010, 2015). On the other hand, the requirement $\inf_{\bthech\in \itA_\epsilon}L(\bthe,\kappa_0)>L(\bthe_0,\kappa_0)$  ensures that $L(\bthe_0,\kappa_0)$ does not   attain a minimum value at infinite. It was also a requirement in Boente \textsl{et al.} (2006) and Boente and Rodr\'{\i}guez (2010) to guarantee strong consistency. It can be replaced by the condition that $(\wbbe, \wblam)$ lie ultimately in a compact set since $(\bbe_0,\eta_0)$ is the unique minimizer of $L(\bbe,g,\kappa_0)$ as stated in \eqref{fisher}.
  
  Assumption   \textbf{C4} is satisfied for most loss functions  $\rho$.  Effectively, assume that $\kappa_0$ is known and that the densities are such that the covering number of the class 
  $$\itF_0=\{g(y,\bx)=\log  f\left(y,H\left(\bx\trasp\bbe+\bla\trasp\bB\right)\right), \bbe\in \real^p, \bla\in \real^{k_n}\}$$
   grows at a polynomial rate, i.e., it is
  bounded by $A \epsilon^{-(k_n+p+1)}$. Then,  if the functions $\varphi(s)$ and  $G(H(s))$ are of bounded variation, we obtain the result using that 
  $N\left(\epsilon,{\cal H}_1+{\cal H}_2,L_r(\qu)\right) \le   N\left(\epsilon/2,\itH_1,L_r(\qu)\right)N\left(\epsilon/2,\itH_2,L_r(\qu)\right)$. A similar bound can be obtained for the bracketing numbers. For the score functions usually  considered in robustness, such as the Tukey's biweight function  or the score function introduced in Croux and    Haesbroeck (2002) for the logistic model, $\varphi$ and $G(H(s))$  have bounded variation and the required condition is easily verified using the  permanence properties of $VC-$classes of functions since  the class $\{ \bx\trasp\bb+\bla\trasp\bB, \bb\in \real^p, \bla\in \real^{k_n}\}$ is a finite--dimensional class and so a $VC-$class. Furthermore, if $\kappa_0$ plays the role of the tuning constant or the scale parameter, as in the Gamma model or when $H(t)=t$ and the errors have a symmetric distribution, the same conclusions hold.
\end{remark}

\subsection{Convergence rates}{\label{sec:tasas}}
In order to derive rates of convergence for the estimators,  we choose as norm $\|\cdot\|_{\itF}$    in the space of functions $\itF$, the $L^{\wp}(Q)$ norm, with $2\le \wp \le \infty$, where $t\sim Q$. Hence, we include as possible norms $\|f\|_{\itF}^2=\|f\|_2^2=\esp f^2(t)$ or  $\|f\|_{\itF}^2=\|f\|_{\infty}$, in which case  $\pi^2(\bthe_1 ,\bthe_2 ) =\|\bbe_1-\bbe_2 \|^2+ \|\eta_1-\eta_2\|^2_{\wp}$ with $\wp=2$ or $\wp=\infty$, respectively. Furthermore, in this setting we define the distance 
$$\pi_\prob^2(\bthe_1 ,\bthe_2 )= \esp\left(w(\bx) \left[\bx\trasp(\bbe_1-\bbe_2)+\eta_1(t)-\eta_2(t)\right]^2\right)\,,$$ 
where for $j=1,2$, $\bthe_j=(\bbe_j, \eta_j)\in \Theta=\real^p\times \itG$. 

We  consider  the following additional assumptions.  Two possible conditions on the bracketing entropy are stated below and according to them weaker or stronger convergence rates are attained.  Conditions under which they hold for some particular models are given in Remark \ref{rem:supC6}. 

To avoid requiring an order of consistency to the estimator $\wkappa$ of $\kappa_0$, from now on we will assume that $L(\bbe_0,\eta_0,a)<L(\bbe,g,a)$ for any $\bbe\in\real^p$ and $g\in \itM_n(\itT_n,\ell)$, $a\in \itV$ such that $(\bbe, g)\ne (\bbe_0,\eta_0)$. This condition clearly entails Fisher--consistency and holds, for instance, for the log--partly linear regression model and when $H(t)=t$ if the errors have a symmetric distribution.

From now on, for $\bla\in \real^{k_n}$, $g_{\blachch}(t)$ stands for the spline function $g_{\blachch}(t)=\bla\trasp\bB(t)$.
\begin{itemize}

\item[\textbf{C5$^{\star}$.}]   Let
$\itG_{n,c, \blach_0}=\{f(y,\bx,t)=\left[\rho\left(y,\bx\trasp\bbe+g_{\blachch}(t), a\right)-\rho\left(y,\bx\trasp\bbe_0+g_{\blachch_0}(t),  a\right)\right]w(\bx)\,,\; \|\bbe -\bbe_0\|<\epsilon_0\,,$
$ \bla\in \itL_{k_n}, a\in  \itV ,  \pi_\prob((\bbe_0,g_{\blachch_0}(t)) ,\bthe) \le c \}$.
For some constant $C_2>0$ independent of $n$, $\bla_0\in \itL_{k_n}$ and $\epsilon$, we have that $N_{[\;]}(\epsilon, \itG_{n,c, \bla_0}, L_2(P))\le  C_2\left(c /\epsilon\right)^{k_n+p+1}$.

\item[\textbf{C5$^{\star\star}$.}] For   $n\ge n_0$, the family of functions    $\itF_{n,c}^\star= \{f(y,\bx,t)=\rho\left(y,\bx\trasp\bbe+g_{\blachch}(t), a\right)w(\bx),  \bla\in \itL_{k_n}, a\in  \itV ,  \pi(\bthe_0 ,\bthe) \le c \}$   is such that for  any    $0<\epsilon<1$, $N_{[\;]}(\epsilon,\itF_{n,c}^\star, L_2(P))\le  C_2/\epsilon^{k_n+p+1}$, for some constant $C_2>0$ independent of $n$ and $\epsilon$. 

  \item[\textbf{C6.}] 
  \begin{itemize}
 \item[a)] The function  $\rho$ is twice continuously differentiable with respect to its second argument with derivatives    $\Psi\left(y,u, a\right)={\partial \rho(y,u,a)}/{\partial u}$ and $\chi\left(y,u, a\right)={\partial \Psi(y,u,a)}/{\partial u} $ such that
    $$\|\Psi\|_{\infty, \itV}=\sup_{y\in \real,u\in \real, a\in \itV} |\Psi\left(y,u, a\right)|<\infty \mbox{ and }\|\chi\|_{\infty, \itV}=\sup_{y\in \real,u\in \real, a\in \itV} |\chi\left(y,u, a\right)|<\infty\,.$$
  \item[b)]  $\esp\left\{\Psi\left(y_1,\bx_1\trasp\bbe_0 + \eta_0(t_1), a\right)| (\bx_1,t_1)\right\}=0$, almost surely, for any $a\in \itV$.
  \end{itemize}
   \item[\textbf{C7.}] $\esp w(\bx_1)\,\|\bx_1\|^2<\infty$.
  \item[\textbf{C8.}] There exists $\epsilon_0>0$ and a positive constant $C_0$, such that for any $\bthe \in \real^p\times  \itM_n(\itT_n,\ell)$ with $\pi^2(\bthe ,\bthe_0 )<\epsilon_0$ and any $a\in \itV$,  $L(\bthe, a)-L(\bthe_0, a)\ge C_0\,\pi_\prob^2(\bthe,\bthe_0)$.

\end{itemize}

\begin{theorem} \label{thm:rate}
  Let $(y_i,\bx_i,t_i)\trasp$ be i.i.d. observations satisfying  \eqref{eq:GPLM} and $k_n = O(n^\nu)$ for $1/(2r + 2) < \nu <  1/(2r)$. Assume that \textbf{C1} to \textbf{C3} and \textbf{C6} to  \textbf{C8} hold and that  $\pi(\wbthe, \bthe_0)\convpp 0$. Then, we have that  
\begin{itemize}
\item[a)] if \textbf{C5$^{\star}$} holds,  $\gamma_n\,\pi_\prob(\bthe_0 ,\wbthe )=O_\prob(1)$, where  $\gamma_n=n^{\min(r\nu, (1-\nu)/2)}$, so if $\nu=1/(1+2r)$, the  estimators converge at the optimal rate $n^{r/(1+2r)}$.
\item[b)] if \textbf{C5$^{\star\star}$} holds,  $\gamma_n\,\pi_\prob(\bthe_0 ,\wbthe )=O_\prob(1)$, for any  $\gamma_n$, such that $\gamma_n \leq O(n^{r\nu})$ and $\gamma_n \log(\gamma_n)\leq O(n^{(1-\nu)/2})$. 
\end{itemize}
\end{theorem}

\begin{remark}{\label{rem:supC6}}
 Note that condition \textbf{C6}b) is analogous to the conditional Fisher--consistency stated in K\@unsch \textsl{et al.} (1989), while condition \textbf{C5$^{\star}$} is analogous to assumption C3$^{\prime}$  in Shen and Wong (1994). Similar arguments to those considered in Shen and Wong (1994) when analysing the Case 3 in page 596, allow to show that  \textbf{C5$^{\star}$} holds, for instance, when $H(t)=t$ when $\phi$ is continuously differentiable with first derivative $\phi^{\prime}$ such that $s \, \phi^{\prime}(s)$ is bounded. It also holds for the logistic model and  for the gamma model when  $w(\bx)\,\|\bx\|^2$ is bounded  using  \textbf{C6}a). 

%For any symmetric matrix $\bA\in \real^{p\times p}$,   denote as $\zeta_1(\bA)\ge \dots \ge \zeta_p(\bA)$ the ordered eigenvalues of $\bA$. It is worth noticing that, if $\zeta_1(\bUpsi)>0$, where $\bUpsi=\esp w(\bx_1)\,\bx_1\,\bx_1\trasp$, $\pi_\prob^2(\bthe,\bthe_0)\le 2 \zeta_1(\bUpsi) \|\bbe-\bbe_0\|^2+\sup_{\bx} w(\bx)\; \|\eta-\eta_0\|_2^2\le A_1 \, \pi^2(\bthe,\bthe_0)$, so if  $L(\bthe, a)-L(\bthe_0, a)\ge A_2\,\pi^2(\bthe,\bthe_0)$, then \textbf{C8} holds with  $C_0=A_1 A_2$. Furthermore, \textbf{C8} holds, for instance, when  $\Psi(y,u,a)$ is continuously differentiable with respect to $u$ and the function $\chi\left(y,u, a\right)={\partial \Psi(y,u,a)}/{\partial u} $ is such that 
%$$\in_{a\in \itV} \inf_{ \bthech \in \Theta_n, \pi^2(\bthech ,\bthech_0 )<\epsilon_0}\inf_{\bx \in \itS_w, t\in \itI} \esp\left(\chi\left(y_1,\bx_1\trasp\bbe +g, a\right)\left|(\bx_1,t_1)=(\bx,t)\right.\right)>0\,,$$
%where $\Theta_n= \real^p\times  \itM_n(\itT_n,\ell)$ and $\itS_w$ is the support of the function $w$.
\end{remark}

%% PV: hasta aca todo ok
\section{The log--Gamma regression model}{\label{gamacaso}}

Among generalized linear models,  the Gamma distribution with a log--link, usually denoted log--Gamma regression, plays an important role, see Chapter 8 of McCullagh and Nelder (1989).  For any $\alpha>0$ and $\mu>0$,   denote as  $\Gamma(\alpha,\mu)$ the parametrization of the Gamma distribution  given by the
density 
$$ f(y,\alpha,\mu)={\alpha^{\alpha}}\;y^{\alpha-1} \;\exp \left(-(\alpha/\mu)y \right)\left
  \{\mu^{\alpha}\,\Gamma(\alpha) \right \}^{-1}\,I_{y \geq
  0}\;.$$
  Under a log--Gamma model,  $y_i|\bx_i\sim
\Gamma(\alpha,\mu_i)$, where $\mu_i=\esp(y_i|(\bx_i,t_i))$ with link function  $\log(\mu_i)= \bbe_0\trasp\bx_i+\eta_0(t_i)$.
As it is well known, in this case, the responses can be transformed so that they are
modelled through a linear regression model with asymmetric errors (see for instance Cantoni and Ronchetti, 2006). Let  $z_i=\log(y_i)$ be the transformed responses,
then  
\begin{equation}
z_i=\bx_i\trasp\bbe_0+\eta_0(t_i) +u_i\,,
\label{eq:modelologgamma}
\end{equation}
 where $u_i$ and $(\bx_i,t_i)$ are independent. Moreover,  $u_i\sim\log(\Gamma(\alpha,1))$ with density 
\begin{equation}
  \label{eq:gammaden}
  g(u,\alpha)=\frac{\alpha^{\alpha}}{\,\Gamma(\alpha)}\;\exp \left [\alpha(u-\exp(u))
  \right ] \;. 
\end{equation}
This density is asymmetric and unimodal with maximum at $u_{0}=0$. For fully parametric linear models. i.e., when $\eta_0(t)=\gamma_0\, t $,    a description on  
robust estimators based on deviances was given in Bianco \textsl{et al.} (2005), while Heritier \textsl{et al.} (2009) considered $M-$type estimators based on Pearson residuals. For the sake of completeness, we will  describe how to adapt the estimators based on deviances to the  present situation.

\noi We will consider the transformed model \eqref{eq:modelologgamma} and denote by $d_{i}(\bbe_0,\eta_0,\alpha)$ the deviance component of the $i$-th observation, i.e., 
$$d_{i}(\bbe_0,\eta_0,\alpha) =2\alpha\;d\left(z_i-\left[\bx_i\trasp\bbe_0+\eta_0(t_i)\right]\right) $$
where   $d(u)= \exp(u)- u -1$. 

In this setting, the classical estimators to be considered below are not based on the quasi--likelihood but on the deviance  and they correspond to the choice $\varphi_a(u)=\varphi(u)=u$ in \eqref{eq:rhobia}, since no tuning constant is needed. Thus, the loss function equals $\rho(z,s)=  d(z- s)$, while
$\Psi(z,s)= {\partial \rho(z,s)}/{\partial s} =1- \exp(z-s)$, $ \chi(z,s)={\partial \Psi(z,s)}/{\partial s}=  \exp(z-s)$. Hence,   if $\bB(t) =\left(B_1(t),\dots,B_{k_n}(t) \right)$,  the classical estimators of $(\bbe_0,\eta_0)$ without any restriction are obtained as $(\wbbe,\weta)$ where $\weta(t)=\wblam\trasp \bB(t)$ with 
$$(\wbbe, \wblam)=\argmin_{\bbech, \blach}\sum_{i=1}^{n} d \left(z_{i}-\left[\bx_{i}\trasp\bbe+\bla\trasp \bB_i\right]\right)\,,$$
where,  for the sake of simplicity, we have denoted as $\bB_i=\left(B_1(t_i), \dots, B_{k_n}(t_i)\right)\trasp$, so    $\bla\trasp \bB_i=\sum_{i=j}^{k_n}\lambda_j B_j(t_i)$.

On the other hand,  robust estimators are obtained controlling large values of the deviance, with a  $\rho-$ function $\phi$,  as defined in Maronna \textsl{et al.} (2006), i.e.,   an even function, non--decreasing on $|y|$, increasing for $y>0$ when $\phi(y)<\lim_{t\to +\infty}\phi(t)$ and such that $\phi(0) = 0$. An example of such functions is   the Tukey's biweight score function, $\phi(y)=\phi_{\tuk }( y) = \min\left(3 y ^2 - 3 y ^4 +  y ^6, 1\right)$. Hence, in this case
$$\rho(z,s, a)  = \phi\left( \frac{\sqrt{d\left(z -s\right)}}a\right)\,,$$
so the tuning constant $a$ needs to be chosen, unless it is fixed by the practitioner. Note that with this notation, the classical estimator corresponds to $\phi(u)=u^2$. 

To provide an algorithm to compute the estimators with an adaptive constant, let us consider the situation in which we have fixed $k_n$ so that we seek for $\bla$ such that $\sum_{i=j}^{k_n}\lambda_j B_j(t)$ provides a good approximation for $\eta_0(t)$.  
As in Bianco \textsl{et al.} (2005), a three step procedure can be considered to compute initial estimators of the parameters. First note that, since the tuning constant of the loss function depends on the unknown parameter $\alpha$,  Bianco \textsl{et al.} (2005) introduce an adaptive sequence of tuning constants $\wc_{\eme,n}$ to define a sequence of $M-$estimators, $\wbthe_{\eme,n}=(\wbbe_{\eme,n},\wblam_{\eme,n})$. 
When $k_n$ is fixed, these estimators, which satisfy  
$$\wbthe_{\eme,n}=\argmin_{\bbech,\blach}\sum_{i=1}^{n}  \phi\left(  \frac{\sqrt{d \left(z_{i}-\left[\bx_{i}\trasp\bbe+\bla\trasp \bB_i\right]\right)}}{\wc_{\eme,n}}\right)\,,$$
for constants  $\wc_{\eme,n}\convprob c_0$, have as asymptotic covariance matrix $\left({B(\phi,\alpha,c_{0})}/{A^{2}(\phi,\alpha,c_{0})}\right)\bSi_0$
where $\bSi_0$ is the asymptotic covariance matrix of the classical estimators obtained when $\phi(u)=u^2$. The constants $B(\phi,\alpha,c_{0})$ and $A^{2}(\phi,\alpha,c_{0})$ depend only on the derivative of  the score function $\phi$ and the shape parameter $\alpha$, but not on the covariates. Hence, the estimators can be calibrated to attain a given efficiency. From now on, denote $C_e(\alpha)$  the value of the tuning constant $c_0$ such that the $M-$ estimator has efficiency $e$ with respect to the classical one. Note that in particular, $e$ will be the efficiency of the regression estimator $\wbbe_{\eme,n}$.

In our modification, we consider the following four step algorithm  to compute a generalized $MM-$estimator. It is worth noticing that the method to be described below is just the proposal considered in  Bianco \textsl{et al.} (2005) applied to the finite--approximation of $\eta_0$ but taking into account the order restrictions.
\begin{itemize}
\item \textbf{Step 1.} We first  compute an initial $S-$estimates   $\wtbthe=(\wtbbe_n,\wtblam_n)$ and the corresponding scale estimate  $\wsigma_{n}$ taking $b=\sup\phi/2$.  To be more precise,   for each value of $(\bbe,\bla)$ let ${\sigma}_{n}(\bbe,\bla)$ be the  $M-$scale estimate of $\sqrt{d \left(z_{i}-\left[\bx_{i}\trasp\bbe+\bla\trasp \bB_i\right]\right)}$ given by 
  \begin{eqnarray*}
    \frac 1n \sum_{i=1}^{n} \phi\left(\frac{\sqrt{d \left(z_{i}-\left[\bx_{i}\trasp\bbe+\bla\trasp \bB_i\right]\right)}}{{\sigma}_{n}(\bbe,\bla)}\right)=b\,
    \,, 
  \end{eqnarray*}
  where $\phi$ is the Tukey bisquare function, $\phi_{\tuk }$.
  \vskip0.1in
  The $S-$estimate of $(\bbe_0, \bla_0)$ for the considered model is defined as $  \;\wtbthe_n=\argmin_{\bbech,\blach}\; {\sigma}_{n}(\bbe,\bla) $ and the corresponding scale estimate by $\wsigma_{n}=\min_{\bbech,\blach}\;\sigma_{n}(\bbe,\bla)$.   Let $u$ be a random variable with density \eqref{eq:gammaden} and write
  $\sigma^{\ast}(\alpha)$ for the solution of
  \begin{eqnarray*}
    \esp_{G}\left[\phi\left(  \frac{\sqrt{d(u_1)}}{\sigma^{\ast}(\alpha)}\right)\right]  =b\, .
  \end{eqnarray*}
  Similar arguments to those considered in  Theorem 5  in Bianco \textsl{et al.} (2005) combined with the results
  of  Theorem \ref{thm:consistency} allow to show that under mild conditions $\wtbbe_{n}\convpp \bbe_0$, $\|\wteta-\eta_0\|^2_{\itF}\convpp
  0$, where $\wteta=\sum_{i=1}^{k_n} \wtlam_i B_i$ and that   $\wsigma_{n}\convpp\sigma^{\ast}(\alpha)$. Moreover, as in Bianco
  \textsl{et al.} (2005),  $\sigma^{\ast}(\alpha)$ is a continuous and   strictly decreasing function and so, an estimator of $\alpha$ can be
  defined as $\widehat{\alpha}_n= \sigma^{\ast -1}(\widehat{\sigma}_{n})$ leading to a  a strongly consistent   estimator for $\alpha$.
\item \textbf{Step 2.} In the second step, we compute   $\wtau_{n}=\sigma^{\ast -1}(\wsigma_{n})$ and
  \begin{eqnarray*}
    \wc_{n}=\max(\wsigma_{n},C_{e}(\wtau_{n}))=\max (\wsigma_{n},C_{e}(\sigma^{\ast -1}(\wsigma_{n})) \, .
  \end{eqnarray*}
  We then have that $\wc_{n}\convprob c_0=\max\{\sigma^{\ast}(\alpha), C_e(\alpha)\}$.
\item \textbf{Step 3.} Let $\wbthe_{n}^{(0)}=\left(\wbbe^{(0)\, {\mbox{\footnotesize \sc t}}},\wblam^{(0)\, {\mbox{\footnotesize \sc t}}} \right)\trasp$  be the adaptive   $MM-$estimator  without restrictions defined by 
  \begin{equation}
    \wbthe_{n}^{(0)}=\argmin_{\bnuch=(\bbech,\blach)}\sum_{i=1}^{n} \phi\left( \frac{\sqrt{d \left(z_{i}-\left[\bx_{i}\trasp\bbe+\bla\trasp \bB_i\right]\right)}}{\wc_{n}}\right) w(\bx_i) . 
    \label{eq:Mestimatormis}
  \end{equation}
  where the weight function $w(\bx)$ controls large leverage points in the $\bx-$covariate space.
\item \textbf{Step 4.} If $\wlam_1^{(0)}\le \wlam_2^{(0)}\le \dots\le \wlam_{k_n}^{(0)}$, the final estimators are $\wbbe=\wbbe^{(0)}$  and $\weta(t)=\sum_{j=1}^{k_n} \wlam_j^{(0)}B_j(t)$. Otherwise, the final estimators are obtained using a standard non--linear minimization  algorithm with restrictions choosing as initial value   $(\wbbe_{n}^{(0)},\bla^{(0)})$, where $\bla^{(0)}\in \itL_{k_n}$. One  possible choice for $\bla^{(0)}$ is $\lambda_1^{(0)}=\lambda_2^{(0)}=0$ and $\lambda_i^{(0)}=i-2$ for $i=3,\dots,k_n$, in which case the  matrix $\bA$ below equals $\bA=(1,-1,0,\dots,0)$.

  We briefly describe below an algorithm to approximate the minimizer of $L_n(\bthe, \wc_{n})$ under the considered restrictions.
  \begin{itemize}
  \item Denote $\wbnabla(\bbe, \bla)=(\wbnabla_1(\bbe, \bla)\trasp, \wbnabla_2(\bbe, \bla)\trasp)\trasp$ the gradient function and $\wbH(\bbe, \bla)=(\wbH_{ij}(\bbe, \bla))_{1\le i,j\le 2}$  the gradient vector and negative Hessian matrix of the objective function, that is,
 \begin{eqnarray*}
 \wbnabla_1(\bbe, \bla)&=& 
  \dst\frac 1n \sum_{i=1}^n \Psi \left (z_i,\bx_i\trasp \bbe +\bB_i\trasp \bla,\wc_{n} \right ) w(\bx_i) \bx_i \\
\wbnabla_2(\bbe, \bla)&=& \dst\frac 1n \sum_{i=1}^n \Psi \left (z_i,\bx_i\trasp \bbe +\bB_i\trasp \bla,\wc_{n} \right ) w(\bx_i) \bB_i 
         \\
\wbH_{11}(\bbe, \bla)&=& 
\dst\frac 1n \sum_{i=1}^n \chi \left (z_i,\bx_i\trasp \bbe+\bB_i\trasp \bla,\wc_{n} \right ) w(\bx_i) \bx_i \bx_i\trasp \\
\wbH_{12}(\bbe, \bla)&=&
          \dst\frac 1n \sum_{i=1}^n \chi \left (z_i,\bx_i\trasp \bbe +\bB_i\trasp \bla,\wc_{n} \right ) w(\bx_i) \bB_i \bx_i\trasp\\
 \wbH_{21}(\bbe, \bla)&=&
          \dst\frac 1n \sum_{i=1}^n \chi \left (z_i,\bx_i\trasp \bbe +\bB_i\trasp \bla,\wc_{n} \right ) w(\bx_i)  \bx_i \bB_i\trasp\\
 \wbH_{22}(\bbe, \bla)&=&\dst \frac 1n \sum_{i=1}^n \chi \left (z_i,\bx_i\trasp \bbe +\bB_i\trasp \bla,\wc_{n} \right ) w(\bx_i)  \bB_i \bB_i\trasp
\end{eqnarray*}
where   
$$\Psi(z,s,a)= {\partial \rho(z,s,a)}/{\partial s} =\frac{1}{2\, a\,\sqrt{d\left(z -s\right)} }\phi^{\prime}\left( \frac{\sqrt{d\left(z -s\right)}}a\right)\left(1- \exp(z-s)\right)$$
 with $\phi^{\prime}$ the first derivative of $\phi$ and $\chi\left(z,u, a\right)={\partial \Psi(z,u,a)}/{\partial u} $.    Let $\itA=\{i_1,\dots, i_m\}$ the set of indices such that $\lambda^{(0)}_{i_j}=\lambda^{(0)}_{i_j+1}$. If $m>0$ define the working matrix as $\bA\in \real^{m\times(k_n+p)}$ in which the $j-$th row is the vector with its $i_j-$th element equal to $1$ and the  $(i_j+1)-$th element equal to $-1$, the remaining ones equal to $0$.
  \item Fix an initial value $\bthe$ (in the first step, $\bthe=(\wbbe_{n}^{(0)},\bla^{(0)})$ and denote $\wbH=\wbH(\bthe)$, $\wbnabla=\wbnabla(\bthe)$.
  \item \textbf{Step 4.1.} Find the feasible direction as 
    $$\betta=\,-\,\left(\identidad- \wbH^{-1}
      \bA\trasp\left(\bA\wbH^{-1}\bA\trasp\right)^{-1}\bA\right)\wbH^{-1}
    \wbnabla$$
  \item \textbf{Step 4.2.} If $\|\betta\|<\epsilon$ for some $\epsilon>0$  small enough, compute the Lagrange multipliers
    $$\bmu=\,-\,\left(\bA\wbH^{-1}\bA\trasp\right)^{-1}\bA \wbH^{-1} \wbnabla$$
    Let $\mu_i$ be the $i-$th component of $\bmu$.
    \begin{itemize}
    \item If $\mu_i \ge 0$, for all $i\in \itA$, then $\wbthe=\bthe$.
    \item If there exists at least one $i\in \itA$ such that $\mu_i < 0$, determine the index corresponding to the largest $\mu_i$ and remove it from $\itA$ and go to  \textbf{S1}.
    \end{itemize}
  \item \textbf{Step 4.3} Compute 
    $$\nu_1=\min_{\eta_i>\eta_{i+1}, i\notin\itA, 1\le i\le k_n-1} \frac{-(\lambda_{i+1}-\lambda_i)}{\eta_{i+1}-\eta_i}$$
    and find the smallest $r$ such that $L_n(\bthe+\,2^{-r} \betta, \wkappa)<L_n(\bthe, \wkappa)$. Then replace $\bthe$ by $\wtbthe=\bthe+\min(2^{-r},\nu_1)\betta,$ update $\itA$ and    $\bA$ and go to \textbf{Step 4.1}.
  \end{itemize}
    
\end{itemize}
 % available at \url{ http://cran.r-project.org/web/packages/alabama/index.html}.

 The following Lemma states the Fisher--consistency of the functionals  related to the estimators $(\wtbbe_{n},\wteta)$ and  $(\wbbe_{n},\weta)$. Its proof is given in the Appendix A and is a consequence of Lemma 1 in Bianco \textsl{et al.} (2005). 

 \begin{lemma}\label{lema:lema1gam}
  \textsl{If the score function   $\phi : \real \to [0, \infty)$
  is a continuous, non-decreasing and even function such that $\phi(0)=0$. Moreover, if $0 \leq s< v$ with
  $\phi(v) < \sup_s \phi(s)$ then $\phi(s)<\phi(v)$. Assume that, for almost any $t_0$,
  $\prob(\bx\trasp \bbe = c\; \cup\; w(\bx)=0|t=t_0)<1$, for any
  $\bbe\in \mathbb{R}^p$, and $c \in \real$, $(\bbe, c)\ne \bcero$. Then, we have that the functionals related to the estimators  $(\wbbe_{n},\weta)$ are  Fisher--consistent. Furthermore, the functionals related to $(\wtbbe_{n},\wteta)$ are Fisher--consistent when \eqref{eq:identifia} holds.} 
\end{lemma} 

\section{Monte Carlo study}\label{montecarlo}
%\subsection{Behavior of the estimators}{\label{monteest}}

In this Section, we summarize the results of a simulation study designed to  compare the performance of the proposed estimators with the classical ones under a  log--Gamma partly linear isotonic regression model. In all Tables, the estimators  in this paper are indicated as \textsc{rob} while  their classical counterparts are indicated as \textsc{cl}, since they correspond to the estimators based on the deviance. To be more precise, the robust estimators correspond to those controlling large values of the deviance as described in Section \ref{gamacaso} and they were computed using the Tukey's
biweight score function. The weight functions $w$ used to control high leverage points was taken $w$ used to control high leverage points was taken as the
Tukey's biweight function with tuning constant $c_w=4.685$ 
\begin{equation}\label{funcionpesow}
  w(x) = \begin{cases}
    \left(1 - \left[\dst\frac{x  - \wmu_n}{c_w\,s_n}\right]^2\right )^{2}     & |x  - \wmu_n|\le c_w s_n \\
    0 & |x  - \wmu_n|\ge c_w\, s_n\;,
  \end{cases}
\end{equation}
with $\wmu_n$ the   median of $x_i$ and $s_n=\mad(x_i)$,  since we have considered $x_i\in \real$. On the other hand, the classical estimators correspond to the choice $\varphi(t)=t$ in \eqref{eq:rhobia} and $w\equiv 1$. 

We have performed $NR=1000$ replications with samples of  size $n=100$. The value of  $k_n$ was chosen as described in Section \ref{sec:BIC}.
The central model denoted $C_0$ in Tables  corresponds to select $(x_i,t_i)$ independent of each other such that $x_i\sim \mbox{N}(0,1)$, $ t_i\sim \itU(0,1)$. The response variable was generated as $y_i|(x_i,t_i)\sim \Gamma(3,\lambda_i)$, where
$$\esp\left(y_i|(x_i,t_i)\right)=\frac{3}{\lambda_i}= \exp\{\beta_{0} x_i + \eta_{0}(t_i) \}\,$$
 with ${\beta_{0}} = 2$.  Hence, the transformed log--Gamma model is
 $$
 z_i= \beta_{0} x_i + \eta_{0}(t_i) +u_i\,,$$
where $u_i\sim \log(\Gamma(3,1))$. Two choices for the nonparametric component have been considered,  ${\eta_{0,1}}(t) = \sin(  \pi t/2)$ and  ${\eta_{0,2}}(t) = \pi\, t+ 0.25\, \sin( 4 \pi t)$ which leads to Models 1 and 2, respectively. 

For each sample generated, we have considered three contaminations labelled $C_1$, $C_2$ and $C_3$ that lead to contaminated samples $(z_{i,c}, x_{i,c}, t_i)$. We have first generated a sample $v_{i} \sim {\cal U}(0,1)$ for $1\leq i \leq n$ and then, we have considered the following contamination scheme:

\begin{itemize}
\item $C_1$ introduces \textsl{bad} high leverage points in the  carriers $x$, without changing the responses already generated,   i.e., $z_{i,c}=z_i$, $1\le i \le n$, while
  $$ x_{i,c}=  
  \begin{cases}
    x_{i} & \text{if } v_{i}\leq0.90\\
      x_{i}^{\star}  & \text{if } v_{i}>0.90\,,    
  \end{cases} 
  $$
   where $x_{i}^{\star}\sim \text{N}\left(5, 1/{16}\right)$.
\item $C_2$ introduces outlying observations in the responses  generated according to the model   but with an incorrect carrier $x$.  
  $$ z_{i,c}=
  \begin{cases}
    z_{i} & \text{if } v_{i}\leq0.90\\
    z_{i}^{\star}& \text{if }v_{i}>0.90 
  \end{cases}  
  $$
  where $z_{i}^{\star}={\beta_{0}} x_i^{\star}  +{\eta_{0}}(t_i)+ u_i^{\star}$ with $ u_i^{\star} \sim \log(\Gamma(3,1))$ and $x_{i}^{\star}$   a new observation from   a $\mbox{N}\left(5,1/{16}\right)$. 
 Note that the carriers are  not contaminated in this situation, i.e., $x_{i,c}=x_i$.
  
\item $C_3$ corresponds to increasing the variance of the carriers $x$   and also to introduce large values on the responses  
 \begin{eqnarray*}
  x_{i,c}=
  \begin{cases}
    x_{i}& \text{if } v_{i}\leq 0.90\\
    x_{i}^{\star} & \text{if }   v_{i}>0.90,
  \end{cases}
 &\qquad   &  z_{i,c}=
  \begin{cases}
    z_{i} & \text{if } v_{i}\leq 0.90\\
    z_{i}^{\star}& \text{if }  v_{i}>0.90\,,
  \end{cases}
\end{eqnarray*}
  where $x_i^{\star}$ is a new observation from a  $\text{N}(0,25)$ and $z_{i}^{\star}=3\,\log(10)+  u_i^{\star}$ with $ u_i^{\star} \sim \log(\Gamma(3,1))$
\end{itemize}
Table \ref{tab:best-gamma2} summarize the  obtained results and report the mean over replication of $\wbeta-\beta_0$, denoted $\mbox{bias}(\wbeta)$, its standard deviation denoted $\sd(\wbeta)$ and the mean square error, that is, the mean over replications of  $(\wbeta-\beta_0)^2$. To study the performance of the estimators of the regression function $\eta_0$, denoted  $\weta$, we have considered the mean square error  ($ \MISE(\weta)$), i.e, the mean over replications of an approximation of the integrated square error (ISE) given by
$$
{\mbox{ISE}}(\widehat{\eta}) = n^{-1} \sum_{i=1}^{n}\,\left[\weta(t_i)-\eta_0(t_i)\right]^2  \,.$$

\small
\begin{table}[ht!]
  \centering
  \setlength{\tabcolsep}{3pt} 
  \begin{tabular}{c|c|rrrr|rrrr}
    \hline                                              
    \rule{0pt}{2ex} & & \multicolumn{4}{c|}{Model 1} & \multicolumn{4}{c}{Model 2} \\
    \hline                                                     
    \rule{0pt}{3ex}   
 & Estimator    & $\mbox{Bias}(\wbeta)$    & $\sd(\wbeta)$ & $\MSE(\wbeta)$ & $\MISE(\weta)$ & $\mbox{mean}(\wbeta)$ & $\sd(\wbeta)$ & $\MSE(\wbeta)$ & $\MISE(\weta)$  \\\hline
    {$C_0$} & \textsc{cl}  & 0.0002   & 0.0608 & 0.0037 & 0.0088  & 0.0000   & 0.0636 & 0.0040  & 0.0324  \\
   			& \textsc{rob} & 0.0021   & 0.0672 & 0.0045 & 0.0096  & 0.0019 & 0.0700 & 0.0049 & 0.0340   \\
    \hline
    {$C_1$} & \textsc{cl}  & -0.5497 & 0.2170  & 0.3492 & 0.0265  & -0.5549 & 0.2215 & 0.3570  & 0.0556  \\
   			& \textsc{rob} & -0.0016 & 0.0706 & 0.0050   & 0.0100 & -0.0020 & 0.0728 & 0.0053  & 0.0344  \\
    \hline
    {$C_2$} & \textsc{cl}  &  -1.8359 & 0.9343 & 4.2426   & 54.3390 & -1.8168 & 0.9665 & 4.2340   & 52.8369  \\
   			& \textsc{rob} & 0.0002 & 0.0711 & 0.0051  & 0.0103 & -0.0001 & 0.0736 & 0.0054   & 0.0348   \\
    \hline
    {$C_3$} & \textsc{cl}  & -1.9400   & 0.2721 & 3.8376 & 15.0401 & -1.9116 & 0.2581 & 3.7207 & 10.1817 \\
   			& \textsc{rob} &0.0043 & 0.0727 & 0.0053  & 0.0146 & 0.0020 & 0.0749 & 0.0056  & 0.0350  \\    
\hline
    \end{tabular}
  \caption{\small\label{tab:best-gamma2} Summary results for the  estimators of $\beta_0$ and  $\eta_0$, under a Gamma model.
    The estimators are obtained when $k_n$ is the data--driven number of knots that minimizes $BIC(k)$.}
\end{table}

The classical estimator shows its sensitivity under all contaminations, the effect being worst in this case on the estimation of the regression function $\eta_0$ when contaminating the responses as in $C_2$ or  $C_3$. For these two contamination the mean square errors of the classical estimators of $\eta_0$ are more than  one thousand times  those obtained by the robust procedure which are quite close to the corresponding ones under $C_0$. On the other hand, contaminating only on the carriers duplicates of the mean square error of the classical estimators $\weta_{\cl}$. Therefore, as expected large responses affect the estimators of the nonparametric component more than leverage points.  It is worth {noting} that for the studied log--Gamma model,  both the bias and the dispersion  of the classical estimators of $\beta_0$ are increased under
$C_2$ enlarging  the mean square error. On the other hand, the increased mean square error obtained under $C_3$ is mainly due to the bias. The effect of the different contaminations is also striking in Figures \ref{fig:boxplot_modelo1_best} and  \ref{fig:boxplot_modelo2_best} which gives the boxplots of $\wbbe$ under Models 1 and 2, respectively. For instance, under $C_1$  and $C_3$, the whole boxplot is under the horizontal line which corresponds to the
true value $\beta_0=2$. On the other hand, the robust estimators are quite stable across all   contaminated scenarios.  Furthermore, the stability of the robust procedure is clearly illustrated in Figure \ref{fig:densi_modelo1_best} which plots the density estimators of $\wbeta_{\cl}$ and  $\wbeta_{\rob}$ under the different contamination schemes. The solid black lines correspond to the uncontaminated samples, while the red dashed, the blue dotted and the maroon dashed-dotted lines to contaminations $C_1$ to $C_3$ respectively. Besides, the dashed  green line corresponds to the normal density with mean 2 and standard deviation equal to 0.0608 and 0.0672 for the classical and robust estimators, respectively. Note that these values correspond to $\sd(\wbbe)$ reported in Table \ref{tab:best-gamma2},  for clean samples. For the robust estimators all the density estimators are over-imposed showing that the contaminations have a mild effect on the estimations. On the other hand, when using the classical procedure based on the deviance, the densities of the estimators computed with contaminated samples move away from that obtained when clean data are considered, leading to unreliable estimates.

\begin{figure}[ht!]
  \centering
  \subfigure[$C_0$]{\includegraphics[width=.45\textwidth]{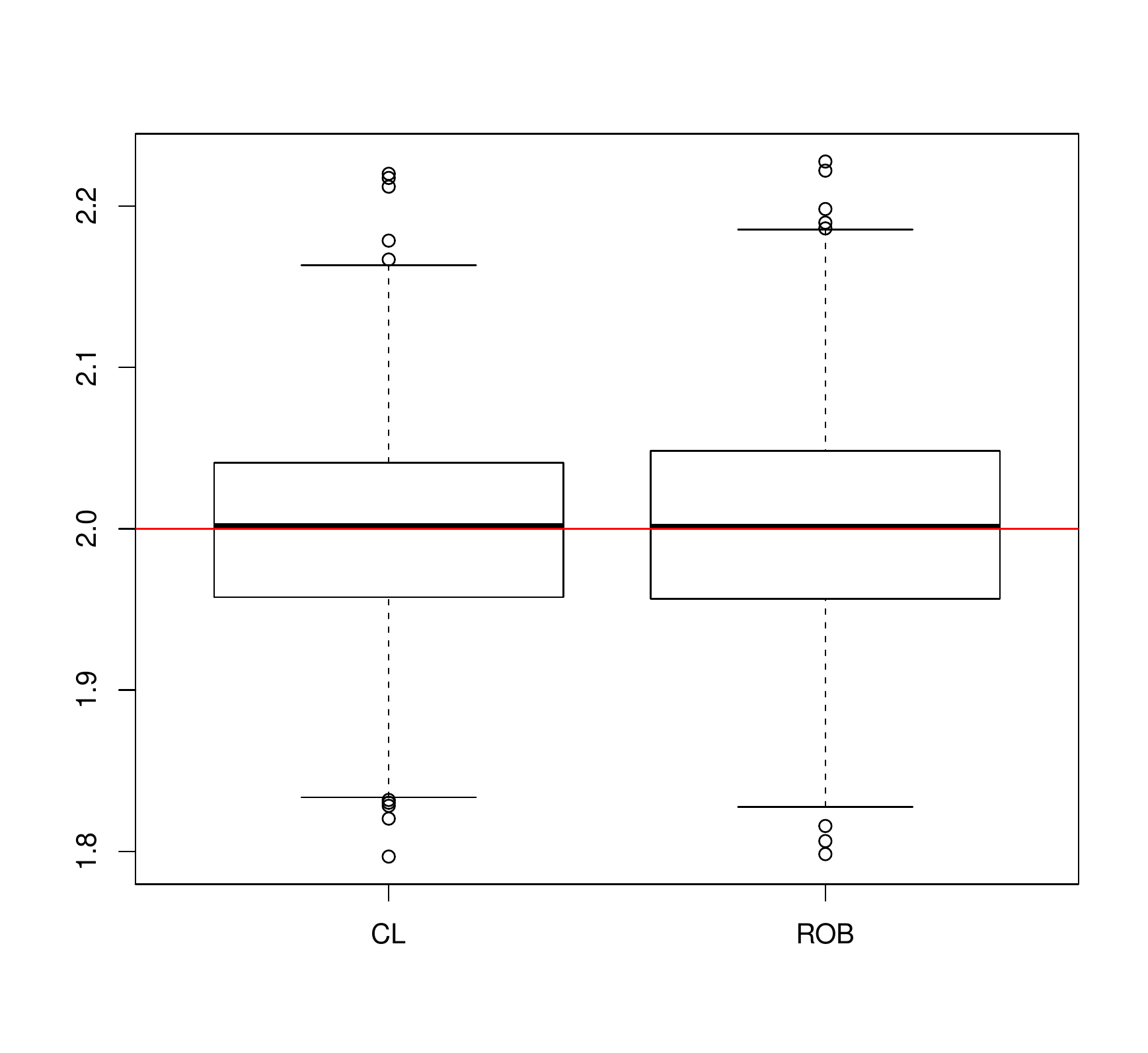}}
  \subfigure[$C_1$]{\includegraphics[width=.45\textwidth]{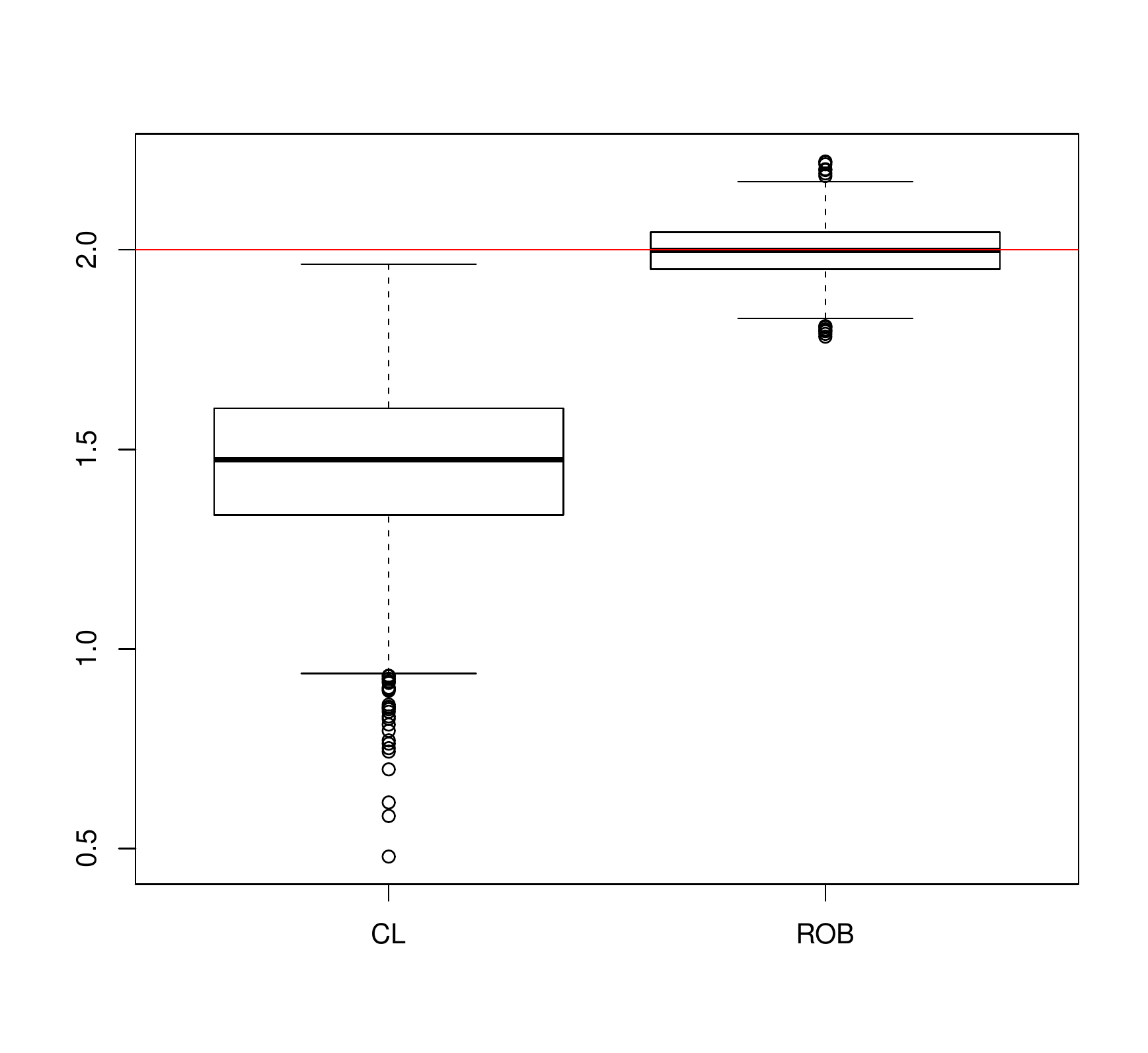}}
  \subfigure[$C_2$]{\includegraphics[width=.45\textwidth]{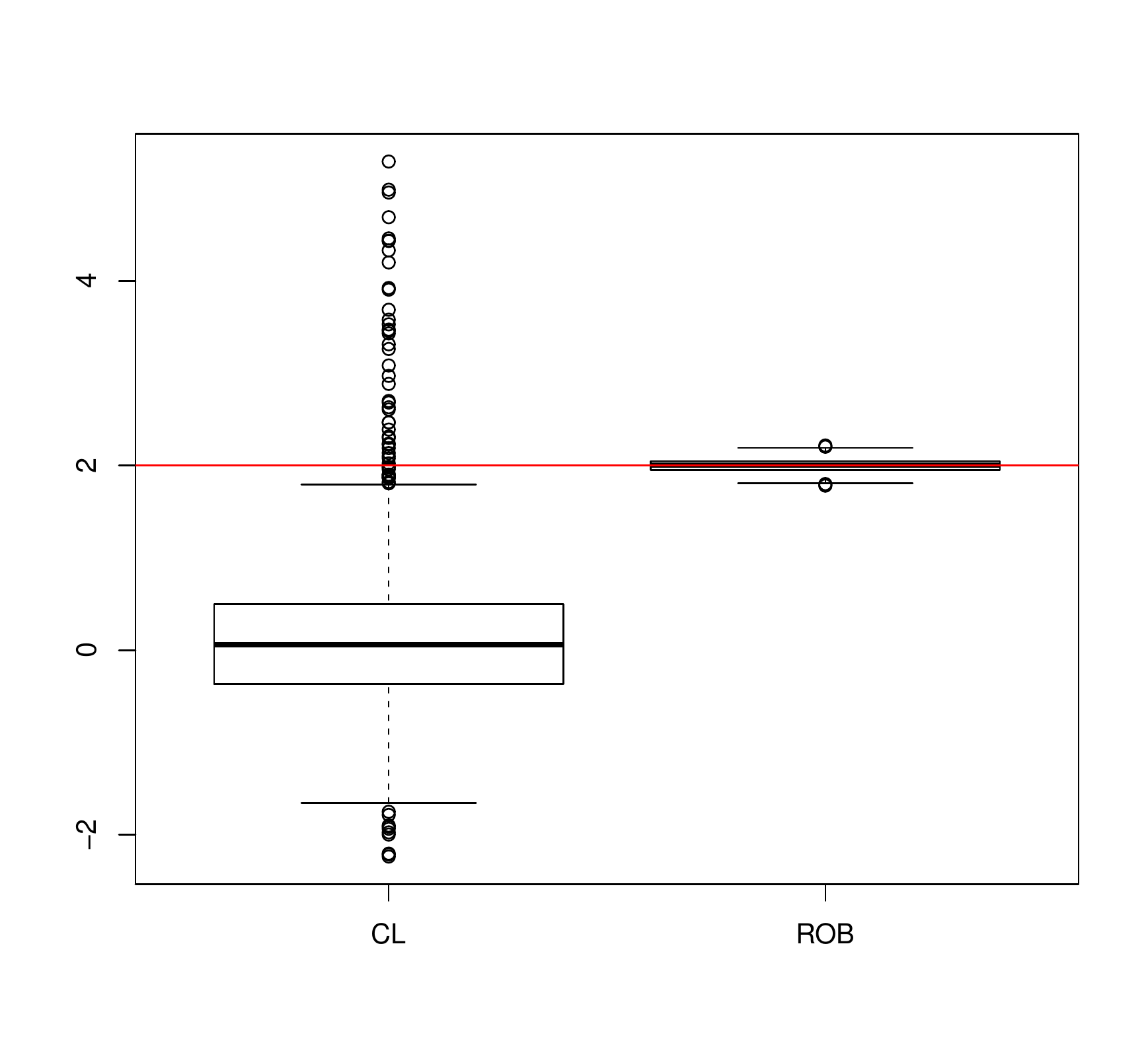}}
  \subfigure[$C_3$]{\includegraphics[width=.45\textwidth]{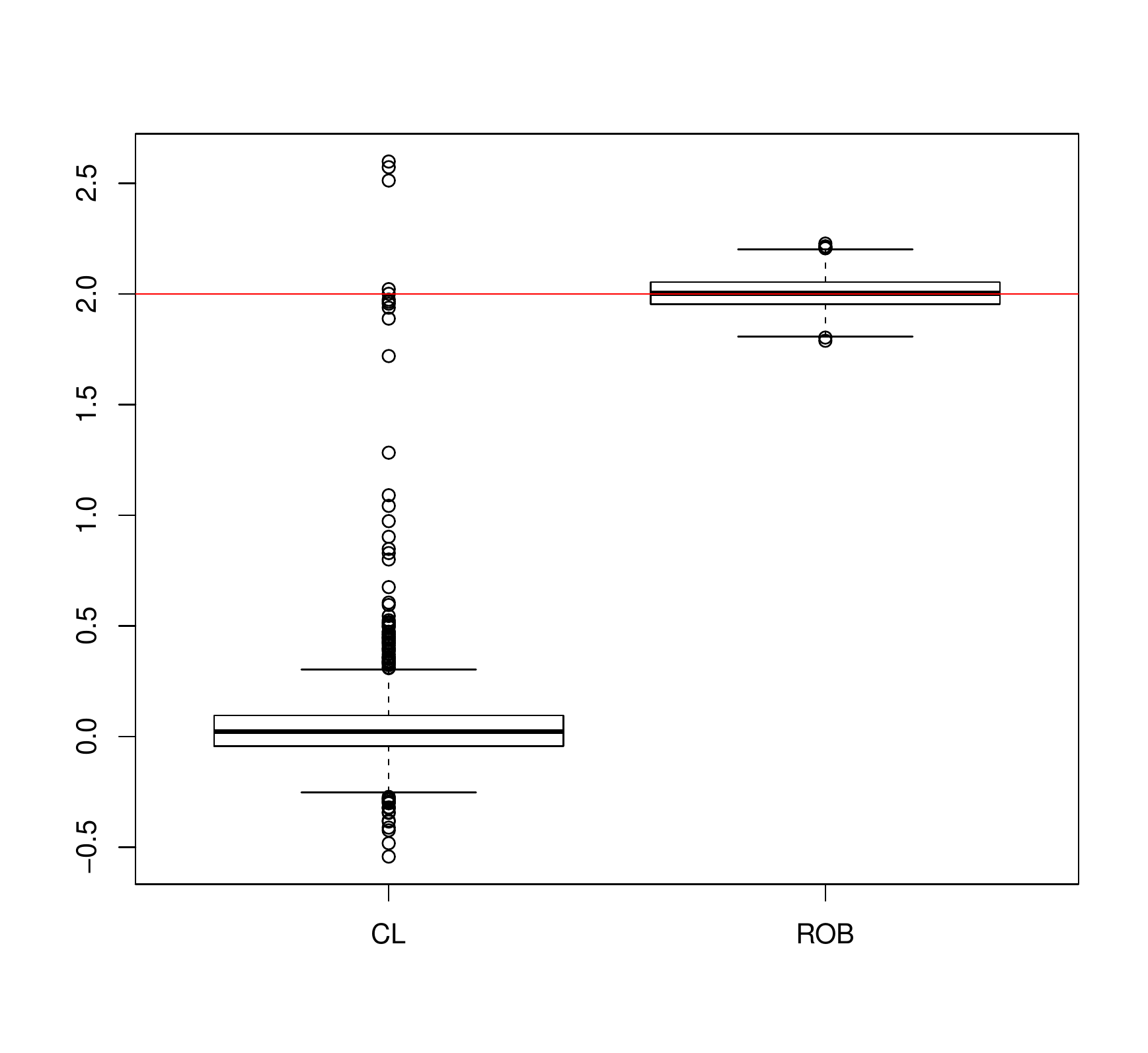}}
  \caption{\small \label{fig:boxplot_modelo1_best} Boxplots of the
    estimators $\wbeta$ of $\beta_0$, under a log--Gamma Model with
    $\eta_0=\eta_{0,1}$.}
\end{figure}

\begin{figure}[ht!]
  \centering
  \subfigure[$C_0$]{\includegraphics[width=.45\textwidth]{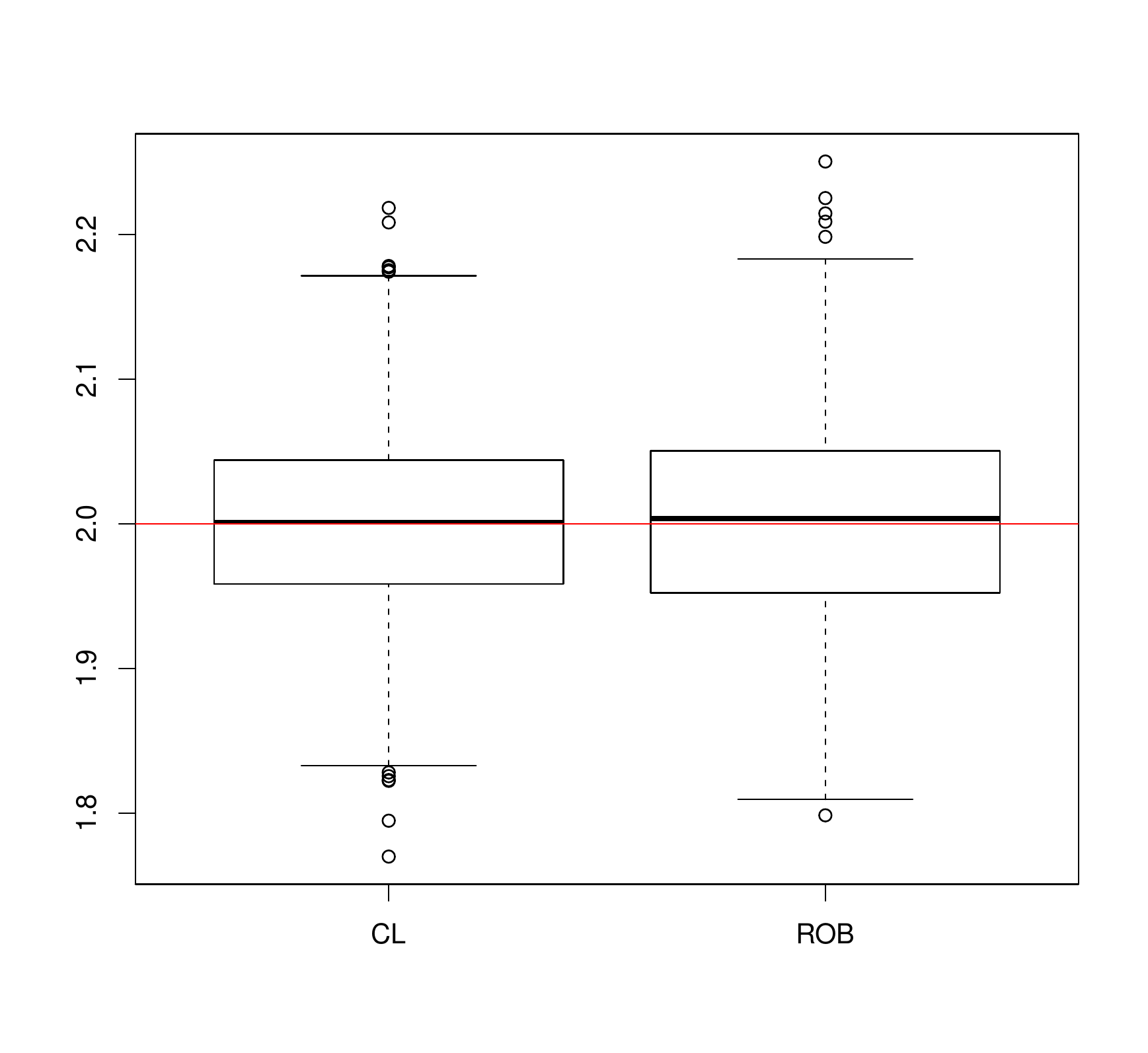}}
  \subfigure[$C_1$]{\includegraphics[width=.45\textwidth]{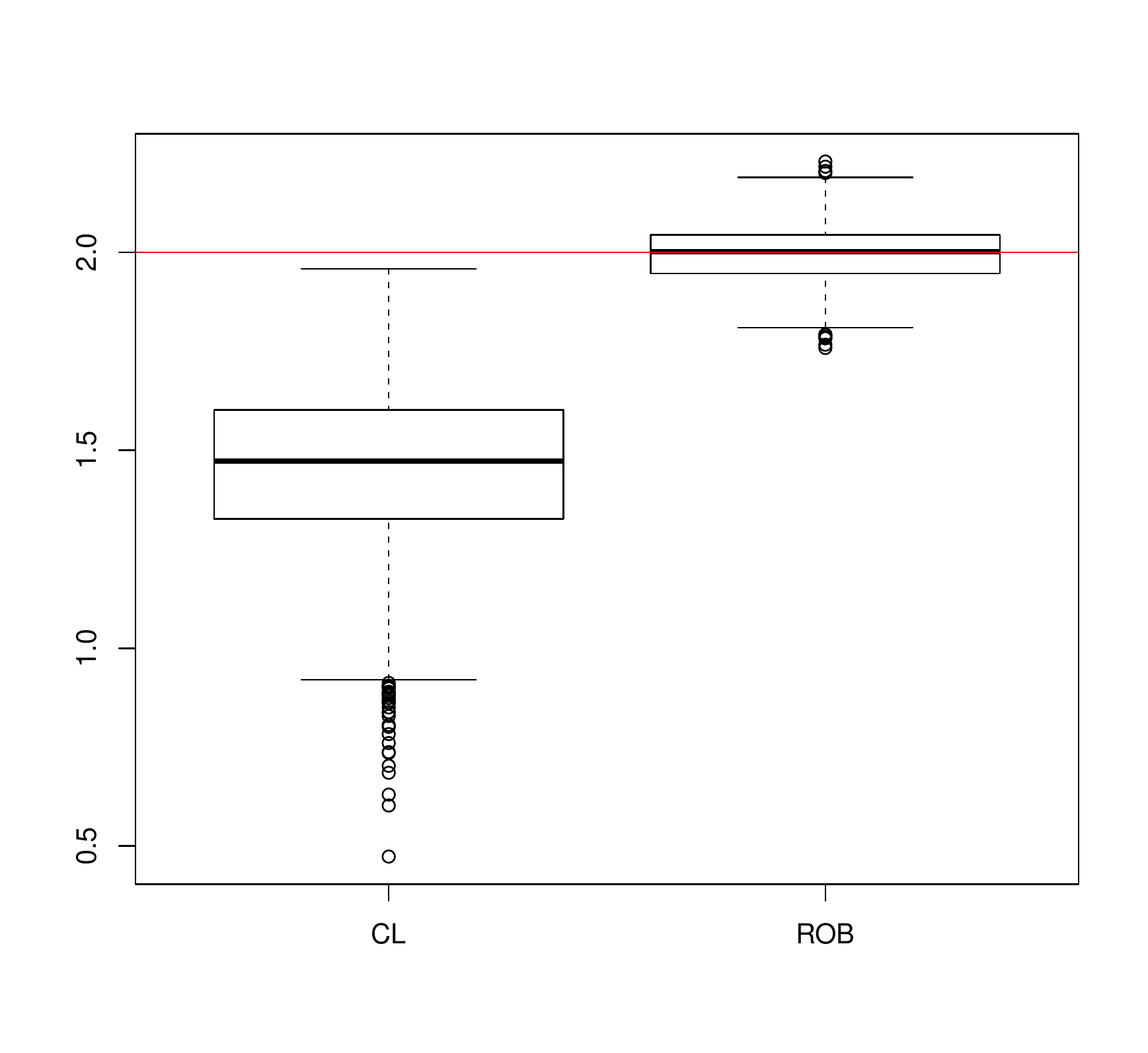}}
  \subfigure[$C_2$]{\includegraphics[width=.45\textwidth]{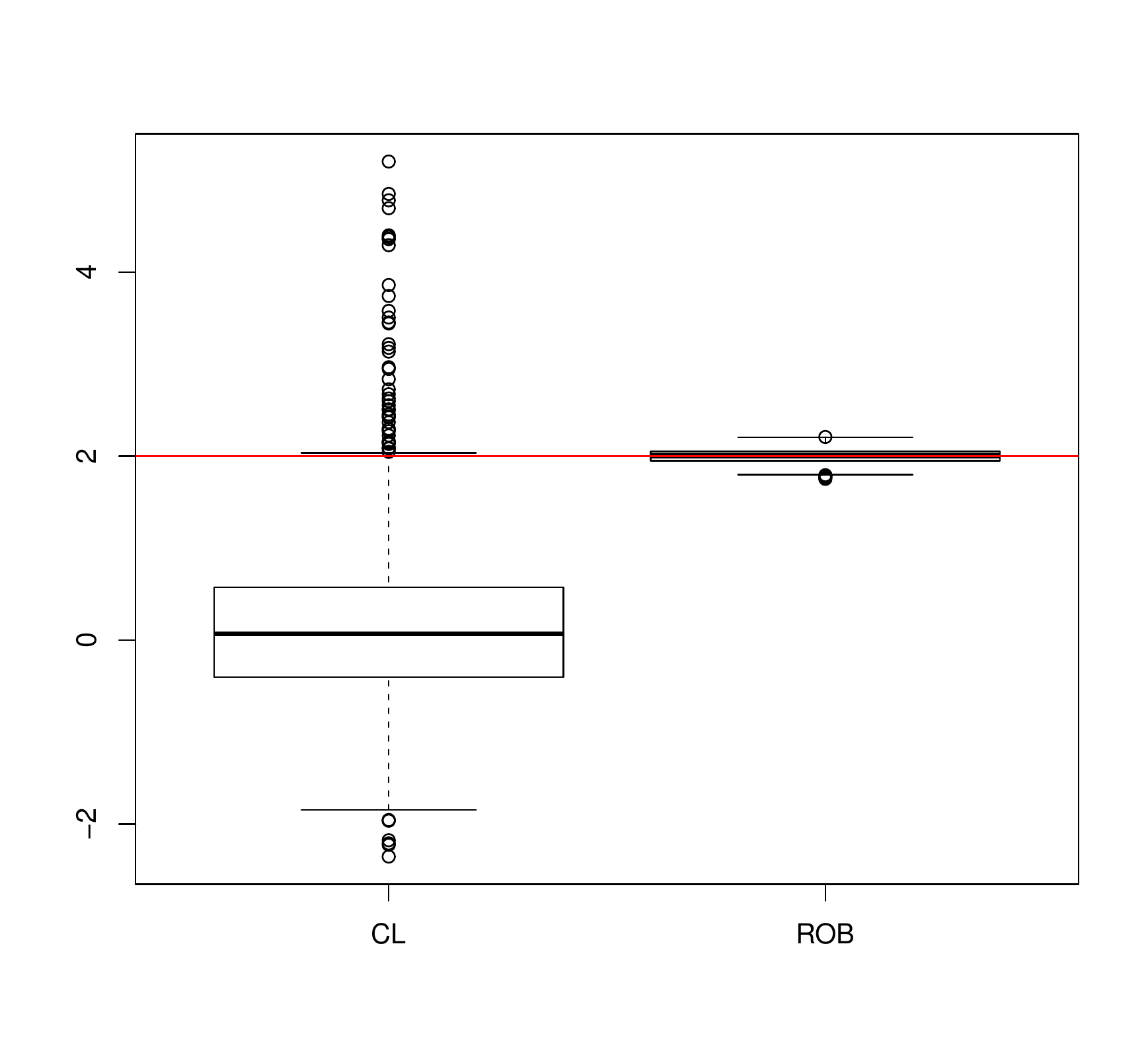}}
  \subfigure[$C_3$]{\includegraphics[width=.45\textwidth]{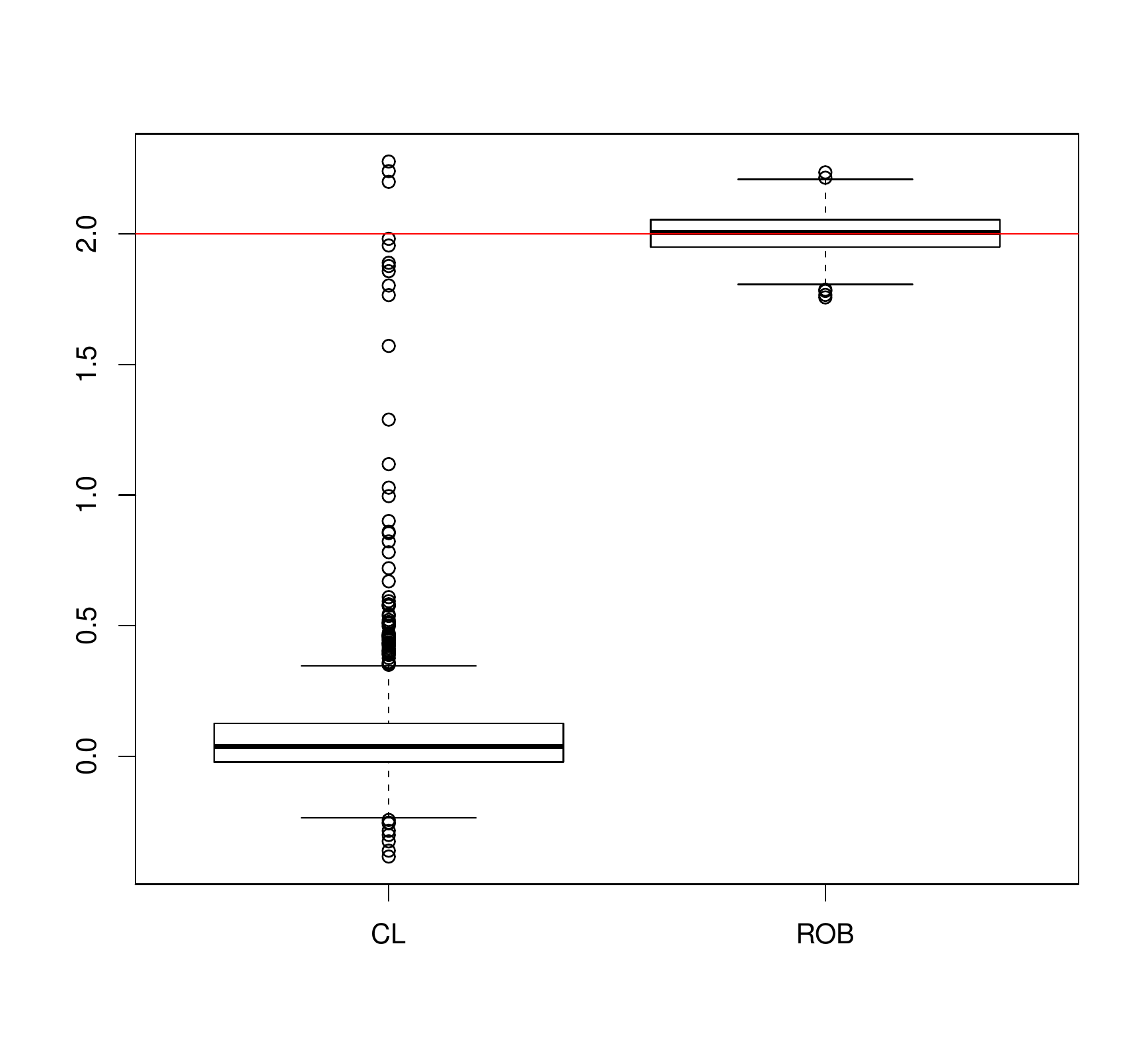}}
  \caption{\small \label{fig:boxplot_modelo2_best} Boxplots of the
    estimators $\wbeta$ of $\beta_0$, under a log--Gamma Model with
    $\eta_0=\eta_{0,2}$.}
\end{figure}

\begin{figure}[ht!]
  \centering
  \subfigure[Classical]{\includegraphics[width=.45\textwidth]{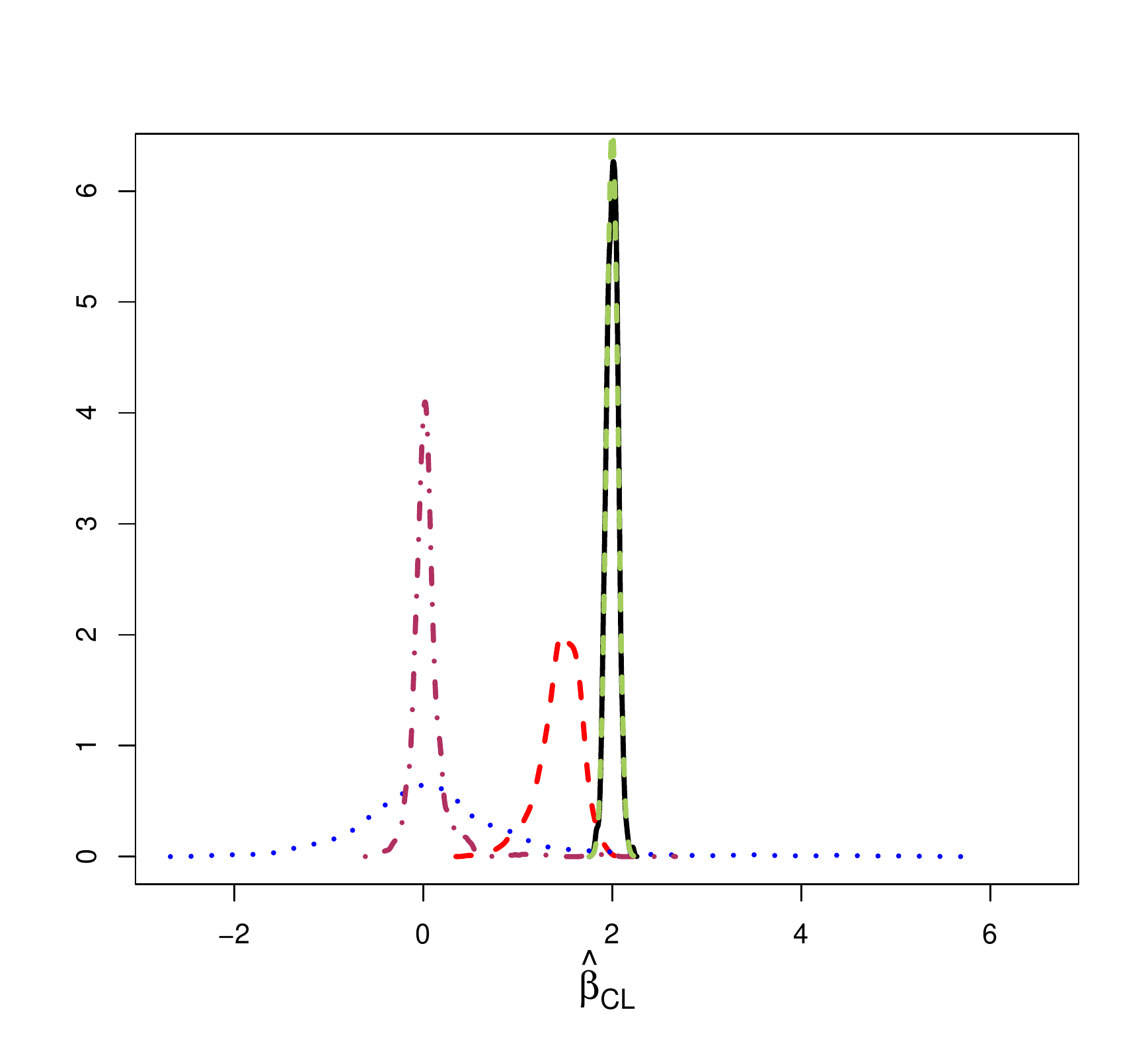}}
  \subfigure[Robust]{\includegraphics[width=.45\textwidth]{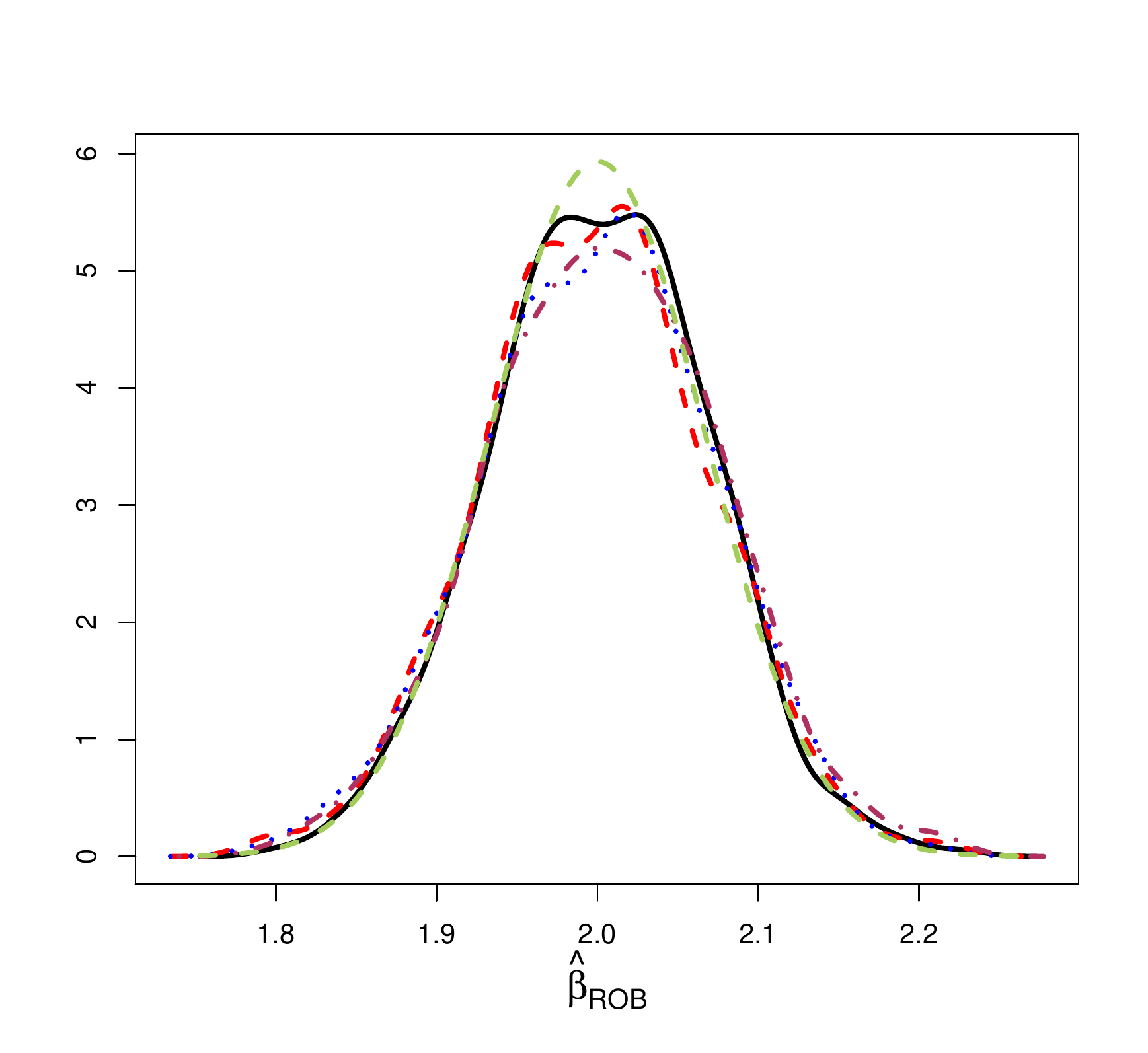}}
  \caption{\small \label{fig:densi_modelo1_best} Density estimator of
    the classical and robust estimators, $\wbeta_{\cl}$ and
    $\wbeta_{\rob}$, of $\beta_0$, under a log--Gamma Model with
    $\eta_0=\eta_{0,1}$.  The solid black lines correspond to the
    uncontaminated samples, while the red dashed, the blue dotted and
    the maroon dashed-dotted lines to contaminations $C_1$ to $C_3$
    respectively.}
\end{figure}

\section{Real data example: Hospital Costs Data}{\label{cost}}
Marazzi and Yohai (2004) introduced a data set that corresponds to the costs of $100$ patients in a Swiss hospital in 1999 for \textit{medical back problems}. They concerned on the relationship between the hospital cost of stay, $y$, (Cost, in Swiss francs) and the following administrative explanatory variables:

\begin{itemize} 
\item $LOS$: length of stay in days
\item $ADM$: admission type (0 = planned; 1 = emergency)
\item $INS$: insurance type (0 = regular; 1 = private)
\item $AGE$: years
\item $SEX$: (0 = female; 1 = male)
\item $DEST$: discharge destination (1 = home; 0 = other)
\end{itemize}

Cantoni and Ronchetti (2006) fitted to the complete data set the model $\log(\esp(y_i|\bx_i))= \bgama_0\trasp\bx_i$ which for Gamma responses is equivalent to
$ z_i=\log(y_{i})= \bgama_0\trasp\bx_i + u_{i}$,
where $u_{i} \sim \log \Gamma(\alpha,1)$ and  $\bx=(ADM, INS , AGE,$ $ SEX,DEST , \log(LOS),1)$. Using their robust proposal, they identified $5$ outliers corresponding to observations labelled as $14, 21, 28, 44$ and $63$, whose weights are less or equal than 0.5. They realized that the atypical points affected the classical estimates of the coefficient of variable $INS$ and the shape parameter. Bianco \textsl{et al.} (2013b) also analysed this data set to perform tests for the covariates $SEX$ and $DEST$.

In this example, we do not impose a linear relation between $z_i=\log(y_i)$ and the $\log LOS$ but we consider the more general isotonic partial linear model
$$z_i=\bbe_0\trasp \bx_i+\eta_0(t_i)+u_i$$
where $u_{i}$ has log $\Gamma(\alpha,1)$ and $\bx=( ADM, INS , AGE, SEX,DEST)$, while $t= \log(LOS)$ and $\eta_0$ is non--decreasing. The monotone assumption on $\eta_0$ is natural in this example, since the hospital cost increases  the longer the stay. The obtained results for the estimators of $\bbe_0$ are reported in  Table \ref{tab:costos-nuevos}. For the classical estimators, denoted $\wbbe_{\cl}$, the $BIC$ criterion selected $k_n=4$, while for the robust ones, denoted $\wbbe_{\rob}$, the best choice was $k_n=5$ and the tuning constant for the $\rho-$function bounding the deviances equal $c_\rho=0.3515$. As in the linear fit, the classical estimator of $\bbe_0$ are very sensitive to the 5 outliers, which were also detected in our study. In particular, the shape parameter and the coefficient related to the insurance type are highly affected. After removing these 5 data points, the classical  estimators   $\wbbe_{\cl}^{{-\{5 \}}}$ are very similar to those  obtained using $\wbbe_{\rob}$, showing the good performance of the robust proposal in presence of outliers. We have computed the jackknife estimators of the standard deviation for the estimators of $\bbe$ which are reported between brackets.  

Figure \ref{fig:costos} shows the plot for the estimators of $\eta_0$ obtained using the classical (in red) and robust estimators (in blue) together with the linear fit provided by $\wbbe_{\gm}$, i.e., $\eta(t)=0.8892 \, t + 7.1268$. The linear fit seems to be a good choice for this data set, however, some discrepancies appear near the boundary which may be caused   by a different shape of the regression function for large values of the $\log(LOS)$. It is worth noting that in this case, the shape of the classical estimator is quite close to that of the robust one and this can be mainly explained by the isotonic structure imposed.

\begin{table}[ht!]
   \centering
   \renewcommand{\arraystretch}{1.2}
  \begin{tabular}{c r c |r c|r c}\hline
    & \multicolumn{2}{c|}{$\wbbe_{\cl}$} & \multicolumn{2}{c|}{$\wbbe_{\cl}^{-\{5\}}$} & \multicolumn{2}{c}{$\wbbe_{\rob}$} \\\hline
   $ADM$  & 0.2148  & (0.0560) & 0.2172  & (0.0418) & 0.1979  & (0.0294)  \\
    $INS$  & 0.0984  & (0.1308) & -0.0324 & (0.0514) & -0.0207 & (0.0407)  \\
    $AGE$  & -0.0009 & (0.0014) & -0.0016 & (0.0010) & -0.0019 & (0.0006) \\
    $SEX$  & 0.1088  & (0.0523) & 0.0820  & (0.0352) & 0.0615  & (0.0329) \\
    $DEST$ & -0.1358 & (0.0585) & -0.1608 & (0.0499) & -0.1673 & (0.0304)  \\
    $\walfa$ & 21.0809 & -      & 45.7560 & -       & 46.0088 & -        \\
        \hline
  \end{tabular}
  \caption{\label{tab:costos-nuevos} Analysis of Hospital Costs data under a log--Gamma isotonic partly linear regression model.}
\end{table}

\begin{figure}[ht!]
   \centering
  \includegraphics[width=.5\textwidth]{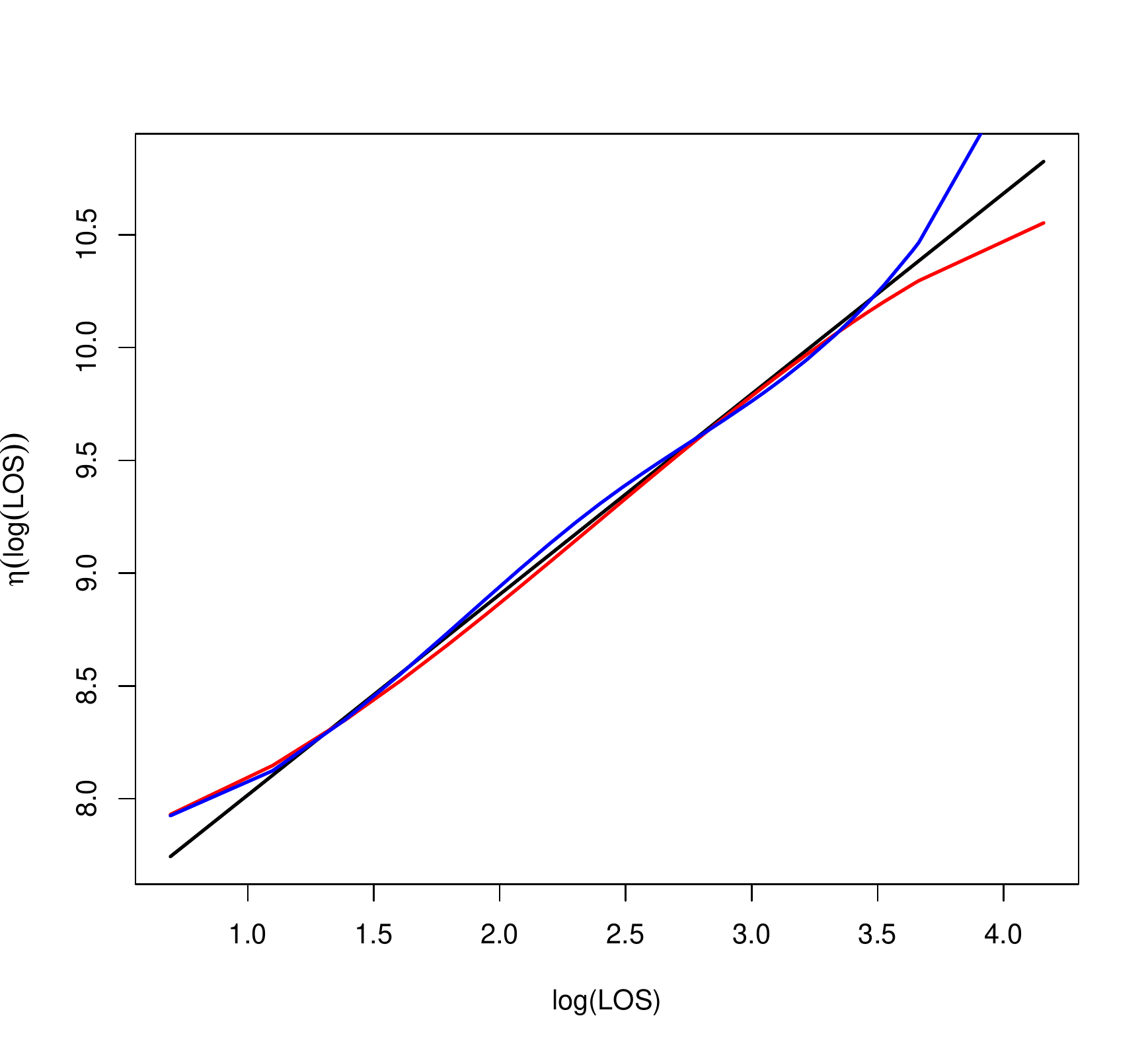} 
  \caption{ \label{fig:costos} Classical (red) and robust (blue) estimators of the  regression function $\eta(t)$ with the linear fit (black).}
\end{figure}

\section{Final comments}{\label{sec:comment}}
 The problem of estimating the nonparametric component   $\eta_0$ and the regression parameter $\bbe_0$ under a generalized partly linear model  has been extensively studied. Among other methods, $B-$splines have been considered to approximate the unknown function $\eta_0$.  One advantage of $B-$splines is that they provide an estimation procedure that can be extended to the situation in which   there are monotone constraints on the nonparametric component by imposing non--decreasing constraints on the coefficients.  To overcome the sensitivity to atypical responses of the classical procedure based on the deviance, we have introduced a family of robust estimators for the   components of a generalized partly linear model
based on monotone $B-$splines, using a bounded loss function   to control large deviance residuals. One of the advantages of our proposal  is that it also allows for an unknown nuisance parameter, such as the scale parameter in partly linear regression models or the shape parameter in a Log--Gamma partly linear regression setting.  Estimation of the nuisance parameter is an important issue since it allows   to calibrate the robust estimators and to down--weight large residuals. Indeed, as in linear regression,  to decide if an observation is an outlier it is necessary to determine  the size of the residuals which strongly depends on the nuisance parameter estimator. 

 The obtained estimators are consistent and rates of convergence are also derived.  The inadequate behaviour of the classical method when atypical data arise in the sample  is   confirmed through our  simulation results. On the other hand, the robust procedure gives more reliable estimators leading to almost  results either under the central log--Gamma model or under the  studied contaminations.

%\appendix

\setcounter{equation}{0}
\renewcommand{\theequation}{A.\arabic{equation}}
\section{Appendix A: Fisher--consistency}

In this section, we discuss conditions ensuring the Fisher--consistency of the proposed estimators, i.e.,  $L(\bbe_0,\eta_0,\kappa_0)=\dst\min_{\bbech\in \real^p, g\in \itG}L(\bbe,g,\kappa_0)$ where $L(\bbe,g,a)$ is defined in (\ref{eq:objfunction}).

\subsection{The logistic case}{\label{sec:Fisherlogit}}
Let us first consider the situation of a logistic partially linear isotonic model. In this case, the loss function $\rho$ given in (\ref{eq:rhobia}) can be written as
\begin{equation}
 \rho(y, u)=y\varphi\left(\,-\,\log\left[H(u)\right]\right)+(1-y)\varphi\left(\,-\,\log\left[1-H(u)\right]\right) + G(H(u)) \,,
\label{phiBY}
\end{equation}
 with $G(t)=G_1(t) + G_1(1 - t)$,  $G_1(t) = \int_0^t \varphi^{\prime}(-\log u) \, du$ and $H(u)= 1/({1+\exp\left(-u\right)})$. 
 
More generally, we have the following results

\begin{lemma}\label{theo:BY_fisher_consistency} 
Let $\rho:\real^2\to \real$ be defined as in    (\ref{phiBY}) where the function  $\varphi:\real_{\ge 0}\to \real$ is such that $\varphi(0) = 0$ and 
\begin{enumerate}
\item[a)] $\varphi: \real_{\geq 0} \to \real_{\geq 0}$ is bounded with continuous and bounded derivative $\varphi^{\prime}$.
\item[b)] $\varphi^{\prime}(t) \geq 0$ and there exists some $c \geq \log 2$ such that $\varphi^{\prime}(t) > 0$ for all $0 < t < c$.
\end{enumerate}
Furthermore, assume that 
 \begin{equation}
       \prob\left(\bx\trasp\bbe=a_0 \cup w(\bx)=0|t=t_0\right)<1, \qquad \forall  (\bbe,a_0)\ne   \bcero \mbox{ and for almost all $t_0$}.
       \label{eq:identifiafuerte}
     \end{equation}
Then,  $(\bbe_0, \eta_0)$ is the unique minimizer of $L(\bbe, g)$.
\end{lemma}

\begin{proof}
The proof is a direct consequence of  Lemma 2.1 in Bianco and Yohai (1996) and \eqref{eq:identifiafuerte}. As in Lemma 2.1 in Bianco and Yohai (1996), let $z$ be a random Bernoulli variable such that $\prob(z=1)=\pi_0$ and define
$$M(\pi_0,\pi)=\esp z \varphi\left(\,-\,\log \pi\right)+(1-z)\varphi\left(\,-\,\log\left[1-\pi\right]\right) + G(\pi)\,.$$
Then we have that $M(\pi_0,\pi_0)<M(\pi_0,\pi)$ for any $\pi\ne \pi_0$.
 Taking conditional expectation, and noticing that $\prob( y=1|(\bx,t))=H(\bx\trasp\bbe_0+\eta_0(t))$, we get that
$$\esp \rho(y,\bx\trasp\bbe+g(t)) w(\bx) = \esp w(\bx) M\left[H(\bx\trasp\bbe_0+\eta_0(t)), H(\bx\trasp\bbe+g(t))\right] \,.$$
For a fixed value $(\bx,t)$, denote $\pi= H(\bx\trasp\bbe+g(t)) $ and $\pi_0=H(\bx\trasp\bbe_0+\eta_0(t))$,  the function $M(\pi_0,\pi)$ reaches its unique minimum when $\pi=\pi_0$ and the proof follows now easily from \eqref{eq:identifiafuerte}.
\end{proof}

 \subsection{The partially linear regression model}{\label{sec:FisherPLM}}
 The partially linear model corresponds to the situation in which the link function equals $H(s)=s$. In this case, the model can be written as
 $$y_i=\bx_i\trasp\bbe_0+\eta_0(t_i)+\sigma_0 u_i\,,$$ 
 where $u_i$ are independent of $(\bx_i,t_i)$ and $\sigma_0$ is the scale parameter. 
 
 As mentioned in Section \ref{sec:lossfunctions}, the loss function may be taken as $\rho(y, u,a)=\phi((y-u)/a)$ for an appropriate function $\phi$. Furthermore, the nuisance parameter  $\kappa_0$ plays the role of the scale parameter. In this section, we consider the situation in which the errors have a   symmetric distribution and the function $\phi$ is an even function. 
 
More precisely, to obtain Fisher--consistency results, we will need the following set of assumptions
\begin{enumerate}[label = \textbf{F\arabic*}]
\item\label{ass:densidad-sym} The random variable $u$ has
  a density function $g_0(u)$ that is even, non-increasing in $|u|$, and
  strictly decreasing for $|u|$ in a neighbourhood of $0$.
\item\label{ass:rho} The function $\phi : \real \to [0, \infty)$
  is a continuous, non-decreasing and even function such that $\phi(0)=0$. Moreover, if $0 \leq s< v$ with
  $\phi(v) < \sup_s \phi(s)$ then $\phi(s)<\phi(v)$. When $\phi$ is
  bounded we assume that $ \sup_s \phi(s)=1$.
\item\label{ass:probaX} For almost any $t_0$,
  $\prob(\bx\trasp \bbe = c\; \cup\; w(\bx)=0|t=t_0)<1$, for any
  $\bbe\in \mathbb{R}^p$, and $c \in \real$, $(\bbe, c)\ne \bcero$.
   
\end{enumerate}
 
\vskip0.1in
The following Lemma entails the Fisher--consistency of the proposed estimators.

\begin{lemma}{\label{lema:lema1}}
  Let $\itG_0=\{ g: [0, 1] \to \real \mbox{ measurable}\}$. Under {\ref{ass:densidad-sym}} to {\ref{ass:probaX}}, we
  have that, for any $\sigma>0$, $(\bbe_0,\eta_0)$ is the unique
  minimizer over $\real^p\times \itG_0$ of
  $$L(\bbe,g, a)=\esp \phi\left(\frac{y - \bx\trasp \bbe -g(t)}{a}\right)w(\bx)\,.$$
\end{lemma}

\begin{proof} Let $\Upsilon(\bx,t)=
  \bx\trasp(\bbe-\bbe_0)+g(t)-\eta_0(t)$, then, we have that
$$L(\bbe,\eta, a)= \esp\phi\left(\frac{
 \sigma_0}a\, u - \frac{\Upsilon(\bx,t)}a\right)w(\bx)\,$$
    Denote as
  $\itA_0 \, = \, \left\{(\bx,t) : \Upsilon(\bx,t)= 0 \right\}$
  and $b(\bx,t)= \Upsilon(\bx,t)/a$. Taking into account that the
  errors are independent of the covariates, we have that
  \begin{align*}
   L(\bbe,\eta, a)
    &= \esp\phi\left(u \frac{\sigma_0}{a}\right)  \,
      \esp\left(w(\bx) \indica_{\itA_0}(\bx,t) \right)  \, + \,
      \esp\left\{\esp\left[ \phi\left(u \frac{\sigma_0}{a}
      -b(\bx,t)\right) \Bigg\vert(\bx,t)\right]w(\bx)\indica_{\itA_0^{c}}(\bx,t)\right\}\,.
  \end{align*}
  Note that $\widetilde{u}=u \sigma_0/a$
  also satisfies \ref{ass:densidad-sym},  hence Lemma 3.1 of Yohai (1987)
  together with \ref{ass:rho} imply for all $b \ne 0$ the following
  strict inequality holds
  \begin{equation}
    \esp \left[\phi \left( u \frac{\sigma_0}{a} -b \right) \right] >
    \esp \left[ \phi \left( u \frac{\sigma_0}{a}  \right) \right]\,.
    \label{eq:cotaesp}
  \end{equation} 
  Then, for any  $(\bx,t) \in \mathcal{A}_0^c$, we get
  \begin{align*}
    \esp\left[\phi\left(u \frac{\sigma_0}{a} -b(\bx,t)\right)
     \Big\vert(\bx,t)=(\bx_0,t_0)\right]
    &= \esp\left[\phi\left(u \frac{\sigma_0}{a}
      -b(\bx_0,t_0)\right)\right] > \esp\left[\phi\left(u
      \frac{\sigma_0}{a} \right) \right]
  \end{align*}
  where the   equality follows from the fact that the errors are
  independent of the covariates. 

  Note that \ref{ass:probaX}   immediately implies that
    $\prob(\mathcal{A}_0^c \cap \{w(\bx)\ne 0\} )> 0$. Then, putting all together, we obtain that
  \begin{align*}
    L(\bbe,\eta, a)
    &= \esp\phi\left(u \frac{\sigma_0}{a}\right)  \,
       \esp\left(w(\bx) \indica_{\itA_0}(\bx,t) \right)  \, + \,
      \esp\left\{\esp\left[ \phi\left(u \frac{\sigma_0}{a}
      -a(\bx,t)\right)  \Bigg\vert(\bx,t)\right]w(\bx)\indica_{\itA_0^{c}}(\bx,t)\right\} \\
    &> \esp \phi\left(u \frac{\sigma_0}{a}\right)  \,
       \esp\left(w(\bx) \indica_{\itA_0}(\bx,t) \right) \, + \,
      \esp\left\{\esp\left[ \phi\left(u \frac{\sigma_0}{\sigma}
      \right) \right]w(\bx)\indica_{\itA_0^{c}}(\bx,t)\right\}
      = \esp\left(\phi\left(u\frac{\sigma_0}{a}\right)w(\bx)
      \right)\\
      &> L(\bbe_0,\eta_0, a)\,,
  \end{align*}
  concluding the proof.
\end{proof}
    
   \subsection{The log--Gamma model}{\label{sec:Fisherloggam}  
Under a generalized partially linear model with responses having a gamma distribution, that is, when $y_i|\bx_i\sim
\Gamma(\alpha,\mu_i)$, with $\mu_i=\esp(y_i|(\bx_i,t_i))$ and  $\log(\mu_i)= \bbe_0\trasp\bx_i+\eta_0(t_i)$, the responses can be transformed as  $z_i=\log(y_i)$ so as to deal with the  regression model with asymmetric errors given by (\ref{eq:modelologgamma}), i.e., 
\begin{equation}
z_i=\bx_i\trasp\bbe_0+\eta_0(t_i) +u_i\,,
\label{modeloplm}
\end{equation}
where $u_i$ and $(\bx_i,t_i)$ are independent. Recall that, under a log--Gamma model, the errors are such that   $u_i\sim\log(\Gamma(\alpha,1))$ and their density is strongly unimodal function. 

In this setting, the loss function equals $\rho(z,s, a)  = \phi\left(  {\sqrt{d\left(z -s\right)}}/a\right)$, where   $d(u)= \exp(u)- u -1$.

We will derive Fisher--consistency results that include other skewed distributions with strongly unimodal densities for the errors. For that reason, we will consider the following additional assumption.

\begin{enumerate}[label = \textbf{F\arabic*}]
\setcounter{enumi}{3}
\item\label{ass:densidad-asym} The random variable $u$ has a density function $g_0(u)$ that is strictly unimodal, continuous and $g_0(u)>0$ for all $u$.
\end{enumerate}

The following lemma gives a stronger result than the one stated in Lemma \ref{lema:lema1gam}, since it shows that for any nuisance parameter the true parameters $(\bbe_0,\eta_0)$ minimize the objective function. This result corresponds to the condition required in Section \ref{sec:tasas} to avoid requiring any consistency order  to the nuisance parameter estimator. 

\begin{lemma}{\label{lema:lema2}}
 Let $\itG_0=\{ g: [0, 1] \to \real \mbox{ measurable}\}$ and  consider the partial linear regression model \eqref{modeloplm}, where  the density of the error $u$ satisfies {\ref{ass:densidad-asym}}. Assume that {\ref{ass:rho}} and {\ref{ass:probaX}} hold, then we have  $(\bbe_0,\eta_0)$ is the unique
  minimizer over $\real^p\times \itG_0$  of
  $$L(\bbe,g, a)= \esp\left[\phi\left(
      \frac{\sqrt{d(z - \bx\trasp \bbe -g(t))}}{a}\right)w(\bx)\right]$$
\end{lemma}

\begin{proof} As above, let $\Upsilon(\bx,t)=
  \bx\trasp(\bbe-\bbe_0)+g(t)-\eta_0(t)$ and
  $\itA_0 \, = \, \left\{(\bx,t) : \Upsilon(\bx,t)=   0 \right\}$. Then, we
  have that
  $$ L(\bbe,g, a)= \esp
  \left(\phi\left(\frac{\sqrt{d(u + \Upsilon(\bx,t))}}{a}\right)\, w(\bx)
  \right)\,.$$
Using that the errors are independent of the
  covariates, we conclude that
  \begin{equation}
   L(\bbe,g, a) = \esp
  \left(\phi\left(\frac{\sqrt{d(u)}}{a}\right)\right)
  \esp\left(w(\bx) \indica_{\itA_0}(\bx,t) \right) + \esp \left \{\esp
    \left[\phi\left(\frac{\sqrt{d(u+\Upsilon(\bx,t))}}{a}\right) \Bigg |
      (\bx,t)\right] w(\bx)\,\indica_{\itA_0^c}(\bx,t) \right\}\,.
      \label{eq:expL}
      \end{equation}
  Taking into account that the errors verify \ref{ass:densidad-asym}, from Lemma 1 in
  Bianco \textsl{et al.} (2005) we may bound the second term in \eqref{eq:expL}. Effectively, for any
  $(\bx,t)\in \itA_0^c$ and for any fixed $a>0$, we get
  $$ \esp 
  \left(\phi\left(\frac{\sqrt{d(u+\Phi(\bx,t))}}{a}\right) \Bigg |
    (\bx,t)\right) > \esp  \left(\phi\left(\frac{\sqrt{d(u)}}{a}\right)
    \Bigg | (\bx,t)\right) = \esp 
  \left(\phi\left(\frac{\sqrt{d(u)}}{a}\right) \right)\,,$$ 
  where the last equality follows from the fact that the errors are
  independent of the covariates. Using \ref{ass:probaX}, we get that the
  strict inequality occurs on a set with positive probability and the
  result follows as in Lemma \ref{lema:lema1}.
\end{proof}
\setcounter{equation}{0}
\renewcommand{\theequation}{B.\arabic{equation}}
\section{Appendix B}
Throughout this section we will denote as $\|\rho\|_{\infty}=\sup_{y \in \real, u\in \real, a\in \itV} \rho(y,u,a)$ and  $\|w\|_{\infty}=\sup_{\bx\in \real^p} w(\bx)$.

\subsection{Proof of Theorem \ref{thm:consistency}.}\label{proofcons}
Let $V_{\bbech,g,a}=\rho\left(y,\bx\trasp\bbe+g(t),a\right)w(\bx) $ and denote as $P$ the probability measure of $(y_1,\bx_1,t_1)$ and as $P_n$
its corresponding empirical measure. Then, $L_n(\bbe,g,a)=P_n V_{\bbech,g,a}$ and $L(\bbe,g,a)=P V_{\bbech,g,a}$. 

Recall that $ \itM_n(\itT_n,\ell)=\left\{\sum_{i=j}^{k_n}\lambda_j B_j:   \quad\lambda_1\leq \dots \leq \lambda_{k_n}\right\} =\left\{\bla\trasp \bB: \bla\in \itL_{k_n} \right\}$. The consistency of $\wkappa$ entails that given any neighbourhood $\itV$ of $\kappa_0$, there exists a null set $\itN_\itV$, such that for $\omega\notin \itN_\itV$, there exists $n_0\in \natu$, such that for all $n\ge n_0$ we have that $ \wkappa\in \itV$.

The proof follows similar steps as those used in the proof of Theorem 5.7 of van der Vaart (1998). Let us begin showing that 
\begin{equation}
  \label{eq:convunif}
  A_n=\sup_{\bbech\in \real^p, g\in\itM_n(\itT_n,\ell), a \in \itV}  |L_n(\bbe,g, a)-L(\bbe,g, a)|\convpp 0\,.
\end{equation}
Note that $A_n=\sup_{f\in\itF_n} (P_n-P)f$, where $\itF_n$ is defined in \textbf{C4}. Furthermore, \textbf{C1} entails that $\sup_{f\in   \itF_n}|f|=\|\rho\|_\infty\|w\|_\infty$ and \textbf{C4} and the fact that   $k_n = O(n^\nu)$ with $\nu < 1/(2r)<1$ imply that
 $$\frac 1n \log N(\epsilon, \itF_n, L_1(P_n))=O_\prob(1)\, \frac{k_n}{n}\log\left(\frac 1\epsilon\right)\convprob 0\,.$$ 
Hence, we get that \eqref{eq:convunif} holds (see, for instance, exercise 3.6 in van der Geer, 2000 with $b_n=\max(1, \|\rho\|_{\infty} \|w\|_{\infty})$).

Since $L(\bthe_0, \kappa_0)=\inf_{\bbech\in \real^p, g\in \itG}L(\bbe,g, \kappa_0)$, where $\bthe_0=(\bbe_0,\eta_0)$, we have that 
\begin{equation}
0\le L(\wbthe, \kappa_0)-L(\bthe_0, \kappa_0)= \sum_{j=1}^3 A_{n,j}\,,
\label{desigL}
\end{equation}
 with $A_{n,1}=L(\wbthe, \wkappa)-L_n(\wbthe, \wkappa)$,
$A_{n,2}=L_n(\wbthe, \wkappa)-L(\bthe_0, \kappa_0)$ and $A_{n,3}=L(\wbthe, \kappa_0)-L(\wbthe, \wkappa)$. Noting that $|A_{n,1}|\le A_n$, we obtain that $A_{n,1}=o_{\as}(1)$. On the other hand, since $L(\wbthe, a)=L^{\star} (\wbbe, \wblam, a)$ the equicontinuity of $L^{\star}$ stated in \textbf{C1} and the consistency of $\wkappa$ entails that  $A_{n,3}=o_{\as}(1)$.

We will now bound $A_{n,2}$.  Using Lemma A1 of Lu \textsl{et al.} (2007), we get that there exists $g_n\in \itM_n(\itT_n,\ell)$ with $\ell\ge r+2$, such that $\|g_n-\eta\|_{\infty}=O(n^{-r\nu} )$, for $1/(2r +2) < \nu < 1/(2r)$. Denote $\bthe_n=(\bbe, g_n)$ and let $S_{n,1}=(P_n-P)V_{\bbech,g_n, \wkappa}$
and $S_{n,2}=L(\bthe_n, \wkappa)-L(\bthe_0, \kappa_0)$. Note that $S_{n,1}\le A_n$, so that from \eqref{eq:convunif}, we get that $S_{n,1}\convpp 0$.
On the other hand, if we write $S_{n,2}=\sum_{j=1}^2 S_{n,2}^{(j)}$ where $S_{n,2}^{(1)}=L(\bthe_n, \wkappa)- L(\bthe_n, \kappa_0)$ and $S_{n,2}^{(2)}=L(\bthe_n, \kappa_0)-L(\bthe_0, \kappa_0)$, the continuity of $\rho$ together with the fact that $\|g_n-\eta\|_{\infty}\to 0$ and the dominated convergence theorem entail that $S_{n,2}^{(2)}\to 0$, while the continuity and boundedness of $\rho$ together with  the consistency of $\wkappa$  leads to $S_{n,2}^{(1)}=o_{\as}(1)$. Hence, $S_{n,j}=o_{\as}(1)$ for $j=1,2$.

Using that $\wbthe$ minimizes $L_n$ over $\real^p\times \itM_n(\itT_n,\ell)$ we obtain that
\begin{equation}
A_{n,2}=  L_n(\wbthe ,\wkappa)-L(\bthe_0, \kappa_0) \leq L_n(\bthe_n, \wkappa)-L(\bthe_0, \kappa_0)=S_{n,1}+S_{n,2} \,.
  \label{eq:desigualdad}
\end{equation} 
Hence, from \eqref{desigL} and \eqref{eq:desigualdad} and using that $A_{n,j}=o_{\as}(1)$, for $j=1,3$ and $  S_{n,j}= o_{\as}(1)$, for $j=1,2$,  we conclude that
$$0\le L(\wbthe, \kappa_0)-L(\bthe_0, \kappa_0)=\sum_{j=1}^3 A_{n,j}\le  o_{\as}(1) $$
so that $ L(\wbthe, \kappa_0)\to L(\bthe_0, \kappa_0)$. The fact that
$\inf_{\pi(\wtbthech,\bthech_0)>\epsilon}L(\wtbthe,\kappa_0)>L(\bthe_0,\kappa_0)$ entails that $\pi(\wtheta,\theta)\convpp 0$, concluding the proof. \square

\subsection{Proof of Theorem {\ref{thm:rate}}}\label{proofrate}
To prove Theorem {\ref{thm:rate} under both sets of assumptions, we will state the common steps at the beginning and we then continue the proof when \textbf{C5$^\star$} or \textbf{C5$^{\star\star}$} hold.

\vskip0.2in
  
We denote $ \Theta_n = \real^p\times \itM_n(\itT_n,\ell)\cap \{\bthe=(\bbe,g)\in \Theta: \pi(\bthe,\bthe_0)<\epsilon_0\}$, where $\Theta =\real^p\times \itG$. Note that, except for a null probability set, $\wbthe \in \Theta_n$, for $n$ large enough. As in the proof of Theorem \ref{thm:consistency}, let $g_n\in \itM_n(\itT_n,\ell)$ with $\ell\ge r+2$, $g_n(t)=\bla_n\trasp \bB(t)$, be such that $\|g_n-\eta_0\|_{\infty}=O(n^{-r\nu} )$, for $1/(2r +2) < \nu < 1/(2r)$ and denote $\bthe_{0,n} = (\bbe_0,g_n)$. 

In order to get the convergence rate  of our estimator $\wbthe = (\wbbe,\weta)$ we will apply Theorem 3.4.1 of van der Vaart and Wellner (1996). For that purpose, following the notation in that Theorem, denote as $M(\bthe)= - L(\bthe, \wkappa)$ and $\EME_n(\bthe)=- L_n(\bthe, \wkappa)$ and for $\bthe\in \Theta_n$, denote $d_n(\bthe, \bthe_0)= \pi_{\prob}(\bthe, \bthe_0)$. Note that the function $M$ is random, due to the nuisance parameter estimator $\wkappa$. Let $\delta_n=A\|\eta_0-g_n\|_{{\itF}}$ , where $A=4\,\sqrt{(C_0 /\|w\|_{\infty}+A_0)/C_0}$ with $A_0= \|w\|_{\infty}   \|\chi\|_{\infty}/2$ and $C_0$ given in \textbf{C8}.

Using that $|(L_n(\bthe, \wkappa)- L(\bthe,\wkappa))-(L_n(\bthe_{0,n}, \wkappa)- L(\bthe_{0,n},\wkappa)) |=  | (\EME_n-M)(\bthe)- (\EME_n-M)(\bthe_{0,n})|$, to make use of Theorem 3.4.1 of van der Vaart and Wellner (1996), we have to show that there exists a function $\phi_n$ such that $\phi_n(\delta)/\delta^\nu$ is decreasing on $(\delta_n, \infty)$ for some $\nu<2$ and that  for any $\delta>\delta_n$, 
\begin{eqnarray}
\sup_{ \bthech\in \Theta_{n,\delta}} L(\bthe_{0,n}, \wkappa)- L(\bthe,\wkappa)=\sup_{ \bthech\in \Theta_{n,\delta}} M(\bthe)- M(\bthe_{0,n}) &\lesssim & -\delta^2
\label{aprobar1}\\
\esp^{*} \sup_{  \bthech\in \Theta_{n, \delta}} \sqrt{n} \left |(L_n(\bthe, \wkappa)- L(\bthe,\wkappa))-(L_n(\bthe_{0,n}, \wkappa)- L(\bthe_{0,n},\wkappa)) \right |  & \lesssim &\phi_n(\delta)
\label{aprobar2}\\
d_n(\wbthe, \bthe_{0,n}) & \convprob & 0
\label{aprobar3}
\end{eqnarray}
where the symbol $\lesssim$ means \textit{less or   equal up to a constant},   $\esp^{*}$ stands for the outer expectation and $\Theta_{n,\delta}=\{\bthe\in \Theta_n: \delta / 2  <  d_n(\bthe,\bthe_{0,n}) \leq \delta\}$.

Assumption \textbf{C8} and the fact that $\wkappa\convpp \kappa_0$ entails that, except for a null probability set,  for any $\bthe\in \Theta_n$, $L(\bthe, \wkappa)-L(\bthe_0, \wkappa)\ge C_0\,\pi_{\prob}^2(\bthe,\bthe_0)$. On the other hand, using \textbf{C6}, we get that
\begin{eqnarray*}
 0\le  L(\bthe_{0,n}, a)-L(\bthe_0, a) &=&   \esp\left\{\esp\left[w(\bx )\Psi(y ,\bx \trasp\bbe_0+\eta_0(t ),a) \left(g_n(t )-\eta_0(t )\right)\left|(\bx ,t )\right.\right]\right\}\\
        & +&\frac 1{2}\;   \esp\left[ w(\bx )\, \chi(y ,\bx \trasp\bbe_0+\wteta(t ), a) \left(g_n(t )-\eta_0(t )\right)^2    \right] \\
  &=& \frac 1{2}\;   \esp\left[ w(\bx )\,\chi(y ,\bx \trasp\bbe_0+\wteta(t ), a) \left(g_n(t )-\eta_0(t) \right) ^2  \right] \\
  &\leq  &\frac 1{2}\;  \|w\|_{\infty}   \|\chi\|_{\infty}\esp\left(  g_n(t)-\eta_0(t)  \right)^2  =  A_0\, \|g_n-\eta_0\|_2^2 {\le  A_0\, \|g_n-\eta_0\|_{ \itF}^2}=O(n^{-2\,r\nu} )\,,
\end{eqnarray*}
where $A_0= \|w\|_{\infty}   \|\chi\|_{\infty}/2$ and $\wteta(t)$ is an intermediate value between $\eta_0(t)$ and $g_n(t)$.
Thus, using that $d_n^2(\bthe,\bthe_{0,n})\le 2 d_n^2(\bthe,\bthe_0)+ 2 d_n^2(\bthe_{0,n},\bthe_0) \le  2 d_n^2(\bthe,\bthe_0) + 2  \|w\|_{\infty}\,\|g_n-\eta_0\|_{2}^2  \le  2 d_n^2(\bthe,\bthe_0) + 2  \|w\|_{\infty}\,\|g_n-\eta_0\|_{{\itF}}^2$ and that  $\delta / 2  <  d_n(\bthe,\bthe_{0,n}) $ we obtain that
\begin{eqnarray*}
 L(\bthe, \wkappa)-  L(\bthe_{0,n}, \wkappa) & \ge  & C_0\,d_n^2(\bthe,\bthe_0)-  A_0\, \|g_n-\eta_0\|_{{\itF}}^2 \ge \frac{C_0}2 d_n^2(\bthe,\bthe_{0,n}) - \left( \frac{C_0}{\|w\|_{\infty}}+A_0\right) \|g_n-\eta_0\|_{{\itF}}^2 \\
 &\ge&  \frac{C_0}8 \delta^2 - \frac{1}{A^2}  \left( \frac{C_0}{\|w\|_{\infty}}+A_0\right) \delta_n^2 =\frac{C_0}8 \delta^2- \frac{C_0}{16} \delta_n^2\ge \frac{C_0}{16} \delta^2\,,
 \end{eqnarray*}
concluding the proof of \eqref{aprobar1}.

We have now to  find $\phi_n(\delta)$ such that $\phi_n(\delta)/\delta$ is decreasing in $\delta$ and \eqref{aprobar2} holds.  
Note that from the consistency of $\wkappa$, we have that, with probability one for $n$ large enough
\begin{eqnarray*}
 \sqrt{n} \left |(L_n(\bthe, \wkappa)- L(\bthe,\wkappa))\right.&& \left.-\;(L_n(\bthe_{0,n}, \wkappa)- L(\bthe_{0,n},\wkappa)) \right | \le \\ 
 &&\sup_{  a \in \itV} \sqrt{n} \left |(L_n(\bthe, a)- L(\bthe,a))-(L_n(\bthe_{0,n}, a)- L(\bthe_{0,n}, a)) \right |\,.
 \end{eqnarray*}
Define the class of functions
$$\itF_{n,\delta} = \{V_{\bthech, a}-V_{\bthech_{0,n}, a}: \frac{\delta}2 \leq d_n(\bthe,\bthe_{0,n}) \leq \delta\,, \bthe\in \Theta_n\,, \, a\in \itV\}= \{V_{\bthech, a}-V_{\bthech_{0,n}, a}:  \bthe\in \Theta_{n,\delta}\,, \, a\in \itV\}\,,$$ 
with $V_{\bthech, a}=\rho\left(y,\bx\trasp\bbe+g(t),a\right)w(\bx) $, for $\bthe=(\bbe,g)$. The inequality \eqref{aprobar2} involves an empirical process indexed by $\itF_{n,\delta}$, since
$$\esp^{*} \sup_{  \bthech\in \Theta_{n,\delta}} \sqrt{n} \left |(L_n(\bthe, \wkappa)- L(\bthe,\wkappa))-(L_n(\bthe_{0,n}, \wkappa)- L(\bthe_{0,n},\wkappa)) \right | \le \esp^{*} \sup_{f\in \itF_{n,\delta}} \sqrt{n} |(P_n-P) f|\,.$$
For any $f\in \itF_{n,\delta} $ we have that $\|f\|_{\infty} \le A_1 = 2 \|\rho\|_{\infty} \|w\|_{\infty}$. Furthermore, if $A_2= \|\psi\|_{\infty}  \|w\|_{\infty}$ using that
$$|V_{\bthech, a}-V_{\bthech_{0,n}, a}| %= |\rho\left(y,\bx\trasp\bbe+g(t),a\right)-\rho\left(y,\bx\trasp\bbe_{0}+g_n(t),a\right)|w(\bx)
\le \|\psi\|_{\infty} w(\bx) |\bx\trasp(\bbe-\bbe_{0}) + g(t)-g_n(t)|\,,$$
and the fact that $\pi_{\prob}(\bthe,\bthe_{0,n})=d_n(\bthe,\bthe_{0,n})\le \delta$,  we get that
$$P f^2\le \|\psi\|_{\infty} \esp\left(w^2(\bx) \left[\bx\trasp(\bbe-\bbe_{0})+g(t)-g_{n}(t)\right]^2\right)\le A_2\,  \pi_{\prob}^2(\bthe,\bthe_{0,n})\le  A_2\, \delta^2\,.$$
  Lemma 3.4.2 van der Vaart and Wellner (1996) leads to
$$\esp^{*} \sup_{f\in \itF_{n,\delta}} \sqrt{n} |(P_n-P) f|\le J_{[\;]}\left( A_2^{1/2}\delta,\itF_{n,\delta}, L_2(P)\right) \left ( 1+ A_1 \frac{J_{[\;]}(A_2^{1/2}\,\delta,\itF_{n,\delta}, L_2(P))}{A_2 \delta^2 \; \sqrt{n}}   \right ) \,,$$ 
where $J_{[\;]}(\delta, \itF, L_2(P)) =\int_0^\delta \sqrt{1+ \log N_{[\;]}(\epsilon, \itF, L_2(P)) } d\epsilon$ is the bracketing integral. 

\vskip0.2in
 
a) Assume now that   \textbf{C5$^{\star}$} holds and note that for any $\bthe=(\bbe,g) \in \Theta_{n,\delta}$, $g$ can be written as $g=\bla\trasp \bB$ for some $\bla\in \itL_{k_n}$, so 
 \begin{eqnarray*}
d_n^2(\bthe, \bthe_{0,n})& = &   \esp\left(w(\bx) \left[\bx\trasp(\bbe-\bbe_{0})+(\bla-\bla_n)\trasp \bB(t)\right]^2\right) \,.
\end{eqnarray*} 
Hence,     $\itF_{n,\delta}\subset \itG_{n,c, \blach_n}$ with $c= \delta$  and the bound given in \textbf{C5$^{\star}$} leads to
$$N_{[\;]}\left( \epsilon,\itF_{n,\delta}, L_2(P)\right)\le C_2  \left(\frac{\delta }{\epsilon}\right)^{k_n+p+1}\,.$$
This implies that 
$$J_{[\;]}( A_2^{1/2}\delta,\itF_{n,\delta}, L_2(P)) \lesssim \delta   \sqrt{k_n+p+1}\,.$$
If we denote $q_n = k_n + p+1$ we obtain that for some constant $A_3$ independent of $n$ and $\delta$,
$$\esp^{*} \sup_{\bthech\in \Theta_{n,\delta}} |\mathbb{G}_n V_{\bthech_{0,n}, \kappa_0}-\mathbb{G}_nV_{\bthech, \kappa_0}| \leq A_3\,\left[\delta \, q_n^{1/2}   + \frac{ q_n  }{ \sqrt{n}}\right]\,.  $$
Choosing
$$\phi_n(\delta)=\delta \, q_n^{1/2}   + \frac{ q_n   }{ \sqrt{n}} \,,$$
we have that $\phi_n(\delta)/\delta$ is decreasing in $\delta$, concluding the proof of \eqref{aprobar2}. 
The fact that $\pi(\wbthe, \bthe_0)\convpp 0$, entails that $\pi_\prob(\wbthe, \bthe_0)\convpp 0$ which together with  $\pi_\prob(\bthe_{0,n}, \bthe_0)\to 0$, leads to \eqref{aprobar3}.

Let  $\gamma_n= O(n^{\min(r\nu,(1-\nu)/2)})$, then   $\gamma_n \lesssim \delta_n^{-1}$, where $\delta_n=A\|\eta_0-g_n\|_{{\itF}}=O(n^{-r\nu})$. We have to show that $\gamma_n^2\phi_n \left(1/{\gamma_n}\right)\lesssim \sqrt{n}$.
 Note that 
 $$\gamma_n^2\phi_n \left(\frac{1}{\gamma_n}\right)=\gamma_n  q_n^{1/2}+ \gamma_n^2\, \frac{ q_n  }{\sqrt{n}} =\sqrt{n}\; a_n(1+a_n)\,\,,$$ 
where $a_n=\gamma_n  q_n^{1/2}/\sqrt{n}$. Hence, to derive that $\gamma_n^2\phi_n \left(1/{\gamma_n}\right)\lesssim \sqrt{n}$, it is enough to show that $a_n=O(1)$, which follows easily since  $k_n=O(n^{\nu})$ and $\gamma_n= O(n^\varsigma)$ with $\varsigma=\min(r\nu,(1-\nu)/2)$. 

Finally, the condition $\EME_n(\wbthe)\ge \EME_n(\bthe_{0,n})-O_{\prob}(\gamma_n^{-2})$ required by Theorem 3.4.1 of van der Vaart and Wellner (1996)  is   trivially fulfilled because $\wbthe_n$ minimizes $L_n(\bthe, \wkappa)$. Hence,   we get that $\gamma_n^2 d_n^2(\bthe_{0,n},\wbthe) = O_{\prob}(1)$.

On the other hand, $d_n(\bthe_{0,n},\bthe_0)\le \|w\|_{\infty}^{1/2} \|g_n-\eta_0\|_{\infty}=O(n^{-r\nu})\le \gamma_n$, which together with  $\gamma_n^2 d_n^2(\bthe_{0,n},\wbthe) = O_{\prob}(1)$ and the triangular inequality leads to  $\gamma_n^2 d_n^2(\bthe_{0},\wbthe) = O_{\prob}(1)$, concluding the proof.

\vskip0.2in

b) We will assume now that \textbf{C5$^{\star\star}$} holds. 
Therefore,   using that any $f\in \itF_{n,\delta}$ can be written as $f=f_1-f_2$ with $f_j\in \itF_{n,\epsilon_0}^\star$ and the bound given in \textbf{C5$^{\star\star}$}, we get that    
$$N_{[\;]}\left( \epsilon,\itF_{n,\delta}, L_2(P)\right)\le C_2^2  \frac{1}{\epsilon^{2(k_n+p+1)}}\,.$$
This implies that 
$$J_{[\;]}( A_2^{1/2}\delta,\itF_{n,\delta}, L_2(P)) \lesssim \delta \log\left(\frac{1}{\delta}\right) \sqrt{k_n+p+1}\,.$$
If we denote $q_n = k_n + p+1$ we obtain
$$\esp \sup_{\bthech\in \Theta_{n,\delta}} |\mathbb{G}_n V_{\bthech_{0,n}, \kappa_0}-\mathbb{G}_nV_{\bthech, \kappa_0}| \leq A\left(q_n^{1/2} \delta \log\left(\frac{1}{\delta}\right) + n^{-1/2} q_n \left[ \log\left(\frac{1}{\delta}\right)\right]^2 \right)\,.$$
Choosing
$$\phi_n(\delta)=q_n^{1/2} \delta \log\left(\frac{1}{\delta}\right) + n^{-1/2} q_n \left[ \log\left(\frac{1}{\delta}\right)\right]^2 \,,$$
we have that $\phi_n(\delta)/\delta$ is decreasing in $\delta$.

Therefore, from Theorem 3.4.1 of van der Vaart and Wellner (1996), we  conclude that $\gamma_n^2 d_n^2(\bthe_{0,n},\wbthe) = O_{\prob}(1)$, where $\gamma_n$ is any sequence satisfying $\gamma_n \lesssim \delta_n^{-1}$  with $\delta_n=\pi(\bthe_0,\bthe_{0,n} )=O(n^{-r\nu})$  and $\gamma_n^2\phi_n \left({1}/{\gamma_n}\right) \leq \sqrt{n}$. The first condition, entails that $\gamma_n \leq O(n^{r\nu})$. The second one, implies that 
$$\gamma_n^2 \left (q_n^{1/2}\gamma_n^{-1} \log(\gamma_n)+ q_n n^{-1/2}[\log(\gamma_n)]^2 \right ) \leq n^{1/2}\,,$$
so  using that $k_n=O(n^{\nu})$ we get that $\gamma_n  \log(\gamma_n)\leq O(n^{(1-\nu)/2})$. 
Finally, the condition $\EME_n(\wbthe)\ge \EME_n(\theta_0)-O_{\prob}(r_n^{-2})$ required by Theorem 3.4.1 of van der Vaart and Wellner (1996)  is   trivially fulfilled because $\wbthe_n$ minimizes $L_n(\bthe, \wkappa)$.

On the other hand, $d_n(\bthe_{0,n},\bthe_0)\le \|w\|_{\infty}^{1/2} \|g_n-\eta_0\|_{\infty}=O(n^{-r\nu})\le \gamma_n$, which together with  $\gamma_n^2 d_n^2(\bthe_{0,n},\wbthe) = O_{\prob}(1)$ and the triangular inequality leads to  $\gamma_n^2 d_n^2(\bthe_{0},\wbthe) = O_{\prob}(1)$. \square

%%% 
%%% BIBLIOGRAFIA
%%% 

%%%%%%%%%%%%%%%%%%%%%%%%%%%%%%%%
% AGRADECIMIENTOS
%%%%%%%%%%%%%%%%%%%%%%%%%%%%

\vskip0.1in
\small
\noi \textbf{Acknowledgements.}    This research was partially supported by Grants  \textsc{pip} 112-201101-00742 from \textsc{conicet},  \textsc{pict} 2014-0351  from \textsc{anpcyt} and  20020130100279\textsc{ba} and 20020120200244\textsc{ba}  from the Universidad de Buenos Aires  at Buenos Aires, Argentina.

\section*{References}
\footnotesize
%\small
 
\begin{description}
\item 
  \'Alvarez, E. and Yohai, J. (2012). $M-$estimators   for isotonic regression. \textsl{J. Statist. Plann. Inf.},   \textbf{142}, 2241-2284.

\item  Bianco, A.  and  Boente, G.  (2004). Robust estimators in semiparametric partly linear regression models. \textsl{J. Statist. Planning and Inference} \textbf{122}, 229-252.

\item Bianco, A, Boente, G. and Rodrigues, I. (2013a). Resistant estimators in Poisson and Gamma models with missing responses and an application to outlier detection. \textsl{J. Multivar. Anal.}, \textbf{114},  209-226.

\item Bianco, A, Boente, G. and Rodrigues, I. (2013b)    Robust tests in generalized linear models with missing responses.  \textsl{Comp. Statist. . Data Anal.}, \textbf{65}, 80-97.

\item
  Bianco, A., Garc\'\i a Ben, M. and Yohai, V. (2005). Robust   estimation for linear regression with asymmetric   errors. \textsl{Canad. J. Statist.}, \textbf{33}, 511-528.
  
  \item
   Bianco, A.  and Yohai, V.  (1996). Robust estimation in the logistic regression model. \textsl{Lecture Notes in Statistics}, \textbf{109}, 17-34. Springer--Verlag, New York.

\item 
  Boente, G., He, X. and Zhou, J. (2006). Robust estimates in   generalized partially linear models. \textsl{Ann. Statist.},  \textbf{34}, 2856-2878.

\item
  Boente, G. and Rodr\'{\i}guez, D. (2010). Robust inference   in generalized partially linear models. \textsl{Comput. Statist.     Data Anal.}, \textbf{54}, 2942-2966. 

\item
  Cantoni, E. and Ronchetti, E. (2001). Robust   inference for generalized linear models.   \textsl{J. Amer. Statist. Assoc.}, \textbf{96}, 1022-1030.

\item 
  Croux, C. and Haesbroeck, G. (2002). Implementing the   Bianco and Yohai estimator for logistic regression.   \textsl{Comp. Statist.  Data Anal.}, \textbf{44}, 273-295.

\item  
  Du, J., Sun, Z. and Xie, T. (2013). $M-$estimation   for the partially linear regression model under monotonic   constraints.  \textsl{Statist.  Prob. Letters}, \textbf{83},   1353-1363.
  
\item
  H\"ardle, W., Liang, H. and Gao,   J. (2000). \textsl{Partially Linear Models}. Physica-Verlag.

\item
  He, X. and Shi, P.  (1996). Bivariate tensor--product   $B-$spline in a partly linear model. \textsl{J. Multivariate   Anal.},  \textbf{58}, 162-181.

\item He, X. and Shi, P.  (1998). Monotone B-Spline smoothing. \textsl{J. Amer. Statist. Assoc.}, \textbf{93}, 643-650.

\item 
  He, X., Zhu, Z. and Fung, W. (2002). Estimation in a   semiparametric model for longitudinal data with unspecified   dependence structure. \textsl{Biometrika}, \textbf{89}, 579-590.

\item
Heritier, S., Cantoni, E., Copt, S. and Victoria--Feser, M.P. (2009). \textsl{Robust Methods in Biostatistics}. Wiley Series in Probability and Statistics. Wiley.

\item
  Huang, J. (2002). A note on estimating a partly   linear model under monotonicity constraints. \textsl{J. Statist. Plann. Inf.}, \textbf{107}, 343-351.

\item K\"unsch, H., Stefanski, L. and Carroll, R. (1989). Conditionally unbiased bounded influence estimation in general regression models with applications to generalized linear models. \textsl{J. Amer. Statist. Assoc.} \textbf{84},  460-466.

\item
  Lu, M., Zhang, Y. and Huang, J. (2007).   Estimation of the mean function with panel count data using  monotone polynomial splines. \textsl{Biometrika}, \textbf{94},   705-718.

\item 
  Lu, M. (2010). Spline-based sieve maximum   likelihood estimation in the partly linear model under   monotonicity constraints. \textsl{J. Multivar. Anal.}, \textbf{101},   2528-2542. 

\item Lu, M. (2015). Spline estimation of generalised monotonic regression. \textsl{J. Nonpar. Statist.}, \textbf{27}, 19-39.

\item  McCullagh, P. and Nelder, J. (1989).  \textsl{Generalized Linear Models}. (2nd ed.) London: Champman and Hall. 

\item Marazzi, A. and Yohai, V. (2004). Adaptively truncated maximum likelihood regression with asymmetric errors. \textsl{J. Statist. Plann. Inference}. \textbf{122}, 271-291.

\item Maronna  R., Martin  D.  and Yohai V. (2006).  \textsl{Robust statistics: Theory and methods} , Wiley, New York.

\item
  Ramsay, J. (1988). Monotone regression splines in  action. \textsl{Statistical Science}, \textbf{3}, 425--441.
  
\item
  Schumaker,L. (1981). \textsl{Spline Functions:    Basic Theory}, Wiley, New York.

\item
  Schwarz, G. (1978). Estimating the dimension of a  model. \textsl{Ann. Statist.}, \textbf{6}, 461-464.

  \item 
  Shen, X., and Wong, W. H. (1994)  Convergence rate of sieve estimates. \textsl{Ann. Statist.}, \textbf{22}, 580-615.

\item Stefanski, L., Carroll, R. and Ruppert, D. (1986). Bounded score functions for generalized linear models. \textsl{Biometrika} \textbf{73}, 413-424.

\item
  Sun,Z., Zhang,Z. and Du,J. (2012).  Semiparametric  analysis of isotonic errors--in--variables regression models  with missing response. \textsl{Communications in Statistics:    Theory and Methods}, \textbf{41}, 2034--2060.
 
\item
  Van der Geer, S. (2000). \textsl{Empirical    Processes in $M-$Estimation},  Cambridge University Press.

\item
  van der Vaart, A. (1998). \textsl{Asymptotic    Statistics}, Cambridge Series in Statistical and Probabilistic  Mathematics. Cambridge University Press.
  
\item
  van der Vaart, A.  and  Wellner, J. (1996). \textsl{Weak Convergence and Empirical Processes. With Applications to Statistics}. Springer--Verlag, New York.
\end{description}

\end{document}